\documentclass[aap]{ximsart}

\usepackage{amsthm,amsmath,natbib}

\startlocaldefs

\usepackage[setbuildercolon, altopenint, noslant]{matmatix}
\usepackage{hyperref}
\usepackage{cleveref}
\usepackage{enumerate}

\newtheorem{theorem}{Theorem}[section]
\newtheorem{proposition}[theorem]{Proposition}
\newtheorem{corollary}[theorem]{Corollary}

\newtheorem{lemma}[theorem]{Lemma}
\newtheorem{example}[theorem]{Example}

\newtheorem{remark}[theorem]{Remark}
\newtheorem{definition}[theorem]{Definition}

\DeclareMathOperator{\interior}{int}
\DeclareMathOperator{\Id}{Id}

\newcommand{\Li}{\mathcal{L}}

\newcommand{\Ri}{\mathcal{R}}
\newcommand{\cS}{\mathcal{S}}

\newcommand{\Bi}{\mathscr{B}}
\newcommand{\Ci}{\mathcal{C}}

\newcommand{\Si}{\mathcal{S}}

\newcommand{\Hi}{\mathcal{H}}
\newcommand{\Ii}{\mathcal{I}}
\newcommand{\Is}{\mathscr{I}}
\newcommand{\Ui}{\mathscr{U}}
\newcommand{\If}{\mathfrak{I}}

\newcommand{\tl}{\tilde{\lambda}}

\DeclareMathOperator{\supp}{supp}
\DeclareMathOperator{\Leb}{Leb}
\newcommand{\Pcoal}{\mathcal{P}_{\textrm{coal}}}

\newcommand{\ball}{\Ui_{\mathrm{b}}}

\newcommand{\tiff}{if\textcompwordmark f }

\usepackage{amssymb,tikz}
\newcommand{\mysetminusD}{\hbox{\tikz{\draw[line width=0.6pt,line cap=round] (3pt,0) -- (0,6pt);}}}
\newcommand{\mysetminusT}{\mysetminusD}
\newcommand{\mysetminusS}{\hbox{\tikz{\draw[line width=0.45pt,line cap=round] (2pt,0) -- (0,4pt);}}}
\newcommand{\mysetminusSS}{\hbox{\tikz{\draw[line width=0.4pt,line cap=round] (1.5pt,0) -- (0,3pt);}}}
\renewcommand{\setminus}{\mathbin{\mathchoice{\mysetminusD}{\mysetminusT}{\mysetminusS}{\mysetminusSS}}}

\endlocaldefs

\begin{document}

\begin{frontmatter}

\title{Exchangeable coalescents, ultrametric spaces, nested interval-partitions: A unifying approach}
\runtitle{Coalescents, ultrametric spaces and combs}


\author{\fnms{F\'elix} \snm{Foutel-Rodier}\ead[label=e1]{felix.foutel-rodier@college-de-france.fr}},
\author{\fnms{Amaury} \snm{Lambert}\ead[label=e2]{amaury.lambert@sorbonne-universite.fr}}
\and
\author{\fnms{Emmanuel} \snm{Schertzer}\ead[label=e3]{emmanuel.schertzer@sorbonne-universite.fr}}

\affiliation{Sorbonne Universit\'e, Coll\`ege de France}

\address{F. \text{Foutel-Rodier}\\
Laboratoire de Probabilit\'es,\\
Statistiques \& Mod\'elisation\\
Sorbonne Universit\'e\\
Box 158\\
4 Place Jussieu\\
75252 PARIS Cedex 05 \\
\printead{e1}}

\address{A. \text{Lambert}\\
Center for Interdisciplinary\\
Research in Biology\\
Coll\`ege de France\\
11, Place Marcelin Berthelot\\
75231 PARIS Cedex 05 \\
\printead{e2}}

\address{E. \text{Schertzer}\\
Center for Interdisciplinary\\ 
Research in Biology\\
Coll\`ege de France\\
11, Place Marcelin Berthelot\\
75231 PARIS Cedex 05 \\
\printead{e3}}

\runauthor{F. Foutel-Rodier, A. Lambert, E. Schertzer}

\begin{abstract}
Kingman's representation theorem \citep{kingman_representation_1978}
states that any exchangeable partition of $\N$ can be represented as a
paintbox based on a random mass-partition.  Similarly, any exchangeable
composition (i.e.\ ordered partition of $\N$) can be represented as a
paintbox based on an interval-partition
\citep{gnedin_representation_1997}. 

Our first main result is that any exchangeable coalescent process (not
necessarily Markovian) can be represented as a paintbox based on a random
non-decreasing process valued in interval-partitions, called nested
interval-partition, generalizing the notion of comb metric space introduced
in~\cite{lambert_comb_2017} to represent compact ultrametric spaces.    

As a special case, we show that any $\Lambda$-coalescent can be obtained
from a paintbox based on a unique random nested interval partition called
$\Lambda$-comb, which is Markovian with explicit transitions. This nested
interval-partition directly relates to the flow of bridges
of~\cite{bertoin_stochastic_2003}. We also display a particularly
simple description of the so-called evolving
coalescent \citep{pfaffelhuber_process_2006} by a comb-valued
Markov process.   

Next, we prove that any ultrametric measure space $U$, under mild
measure-theoretic assumptions on $U$, is the leaf set of a tree
composed of a separable subtree called the backbone, on which are grafted
additional subtrees, which act as star-trees from the standpoint of sampling.
Displaying this so-called weak isometry requires us to extend the
Gromov-weak topology of~\cite{greven_convergence_2006}, that was initially
designed for separable metric spaces, to non-separable ultrametric spaces. It
allows us to show that for any such ultrametric space $U$, there is a nested
interval-partition which is 1) indistinguishable from $U$ in the Gromov-weak
topology; 2) weakly isometric to $U$ if $U$ has a complete backbone; 3) isometric
to $U$ if $U$ is complete and separable.  
\end{abstract}

\begin{keyword}[class=MSC]
\kwd[Primary ]{60G09}
\kwd[; secondary ]{60J35}
\kwd{60C05}
\kwd{54E70.}
\end{keyword}

\begin{keyword}
\kwd{Combs}
\kwd{compositions}
\kwd{nested compositions}
\kwd{Lambda-coalescents}
\kwd{flow of bridges}
\kwd{metric measure spaces}
\kwd{Gromov-weak topology.}
\end{keyword}

\end{frontmatter}

\newpage

\section{Introduction}

    \subsection{Ultrametric spaces and exchangeable coalescents} 
    \label{SS:intro}

In this paper we extend earlier work from~\cite{lambert_comb_2017} on the
comb representation of ultrametric spaces. An ultrametric space is a
metric space $(U,d)$ such that the metric $d$ fulfills the additional
assumption
\[
    \forall x,y,z \in U,\; d(x,y) \le \max(d(x,z), d(z,y)).
\]

In applications, ultrametric spaces are used to model the genealogy of
entities co-existing at the same time. The distance between two points
$x$ and $y$ of an ultrametric space is interpreted as the time to the
most recent common ancestor (MRCA) of $x$ and $y$.  For instance, in
population genetics ultrametric spaces model the genealogy of homologous
genes in a population. Another example can be found in phylogenetics
where ultrametric spaces are used to model the evolutionary relationships
between species. 

In population genetics and more generally in biology we do not have 
access to the entire population (that is to the entire ultrametric space)
but only to a sample from the population. To model the procedure of 
sampling we equip the ultrametric space with a probability measure $\mu$
(also referred to as the sampling measure), yielding the notion of 
ultrametric measure spaces.

\begin{definition} \label{def:UMS}
    A quadruple $(U, d, \Ui, \mu)$ is called an ultrametric measure space
    (UMS) if the following holds.
    \begin{enumerate}
        \item[\rm(i)] The distance $d$ is an ultrametric on $U$ which is $\Ui \otimes \Ui$
            measurable.
        \item[\rm(ii)] The measure $\mu$ is a probability measure defined on $\Ui$.
        \item[\rm(iii)] The family $\Ui$ is a $\sigma$-field such that
            \[
                \forall x \in U,\, \forall t > 0,\; 
                \Set{y \in U\suchthat d(x,y) < t} \in \Ui
            \]
            and $\Ui \subseteq \Bi(U)$, where $\Bi(U)$ is the usual Borel
            $\sigma$-field of $(U, d)$.
    \end{enumerate}
    If $\Ui = \Bi(U)$, we say that $(U, d, \Ui, \mu)$ is a Borel UMS.
\end{definition}

\begin{remark}
    This definition might be surprising as we would naively expect a UMS
    to be any ultrametric space with a probability measure on its Borel
    $\sigma$-field. However the previous naive definition is not
    satisfying for several reasons, that are exposed in
    \Cref{SS:generalUMS}. Notice that if $(U, d)$ is separable, then 
    $\Ui = \Bi(U)$ and point {\rm(i)} always holds. We thus recover the usual
        definition of an ultrametric measure space.
\end{remark}

A sample from a UMS is an i.i.d.\ sequence $(X_i)_{i \ge 1}$
distributed according to $\mu$.  The genealogy of the sample is usually
encoded as a partition-valued process, $(\Pi_t)_{t \ge 0}$ called a
\emph{coalescent}. For any time $t \ge 0$, the blocks of the partition
$\Pi_t$ are given by the following relation
\begin{align} \label{eq:UMSsampling}
    i \sim_{\Pi_t} j \iff d(X_i, X_j) \le t. 
\end{align}
The process $(\Pi_t)_{t \ge 0}$ has two major features. First 
a well-known characteristic of ultrametric spaces is that for a given
$t$ the balls of radius $t$ form a partition of the space that gets coarser
as $t$ increases. This implies that given $s \le t$, the 
partition $\Pi_t$ is coarser than $\Pi_s$. Second, if $\sigma$ denotes
a finite permutation of $\N$ and $\sigma(\Pi_t)$ is the partition
of $\N$ whose blocks are the images by $\sigma$ of the blocks of $\Pi_t$,
we have 
\[
    (\Pi_t)_{t \ge 0} \overset{\text{(d)}}{=} (\sigma(\Pi_t))_{t \ge 0}.
\]
We call any c\`adl\`ag partition valued process that fulfills these two
conditions an \emph{exchangeable coalescent} (note that the process
$(\Pi_t)_{t \ge 0}$ is not necessarily Markovian).

    \subsection{Combs in the compact case}
    \label{SS:introComb}

\begin{figure}
    \begin{center}  
    \includegraphics[width=\textwidth]{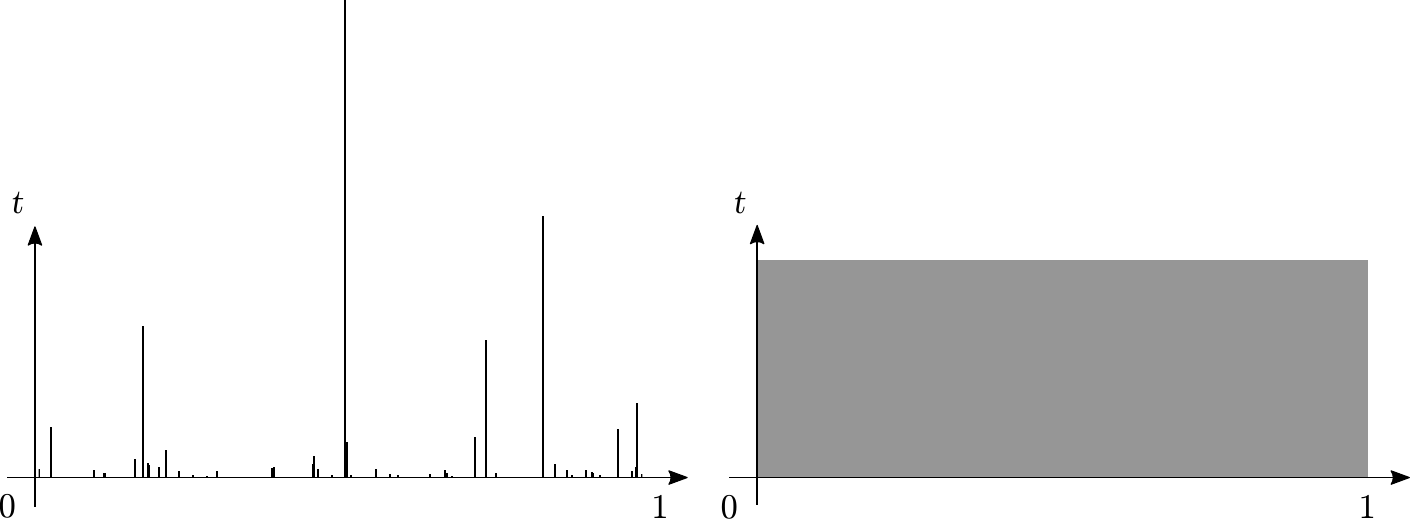}
    \caption{Representation of two nested interval-partitions. A 
    point $(x,t)$ is plotted in dark if $x \not \in I_t$. 
    Left panel: A realization of the Kingman comb, a tooth of
    size $y$ at location $x$ represents that $f(x) = y$. 
    Right panel: The star-tree comb, an example of a nested interval-partition
    that cannot be represented as an original comb.}
    \label{F:kingman}
    \end{center}
\end{figure}
\paragraph{Combs and ultrametric spaces}
In this section, we address similar questions in the much simpler
framework of comb metric spaces which have been introduced recently by
\cite{lambert_comb_2017} to represent \textit{compact} ultrametric
spaces.  A comb is a function
\[
    f \colon [0,1] \to \R_+
\]
such that for any $\epsilon > 0$ the set $\Set{f \ge \epsilon}$
is finite (see \Cref{F:kingman} left panel). To any comb is associated a comb metric $d_f$
on $[0, 1]$ defined as
\[
    \forall x,y \in [0,1],\; \ d_f(x,y) = \Indic{x \ne y} 
    \sup_{\ClosedInterval{x\vee y, x\wedge y}} f.
\]
In general $d_f$ is only a pseudo-metric on $[0, 1]$ and it is 
easy to verify that it is actually ultrametric. 
One of the main results in \cite{lambert_comb_2017} shows that any
compact ultrametric space is isometric to a properly completed and quotiented
comb metric space (see Theorem~3.1 in \cite{lambert_comb_2017}).

\paragraph{Exchangeable coalescents} 
We also will be interested in the
relation between combs and exchangeable coalescents. Any
comb metric space $([0,1], d_f)$ can be naturally endowed with the 
Lebesgue measure on $[0, 1]$. Sampling from a comb can be seen as a 
direct extension of Kingman's paintbox procedure. More precisely, given a comb $f$,
we can generate an exchangeable coalescent $\left(\Pi_t\right)_{t \ge 0}$ by throwing 
i.i.d.\ uniform random variables $(X_i)_{i \ge 1}$ on $[0,1]$ 
and declaring that 
\[
    i \sim_{\Pi_t} j \iff \sup_{[X_i \wedge X_j , X_i \vee X_j]} f \le t. 
\]

For the sake of illustration, we recall the comb representation of the
Kingman coalescent stated in~\cite{kingman_coalescent_1982}. The Kingman
comb is constructed out of an i.i.d.\ sequence $(e_i)_{i \ge 1}$ of
exponential variables with parameter $1$, and of an independent i.i.d.\ sequence
$(U_i)_{i \ge 1}$ of uniform variables on $[0, 1]$. We define the sequence
$(T_i)_{i \ge 2}$ as
\[
    T_i = \sum_{j \ge i} \frac{2}{j(j-1)} e_j.
\]
The Kingman comb $f_K$ is defined as
\[
    f_K = \sum_{i \ge 2} T_i \Indic*{U_i}.
\]
See \Cref{F:kingman} left panel for an illustration of a realization of the
Kingman comb. The paintbox based on $f_K$ is a version of 
the Kingman coalescent (see Section~4.1.3 of \cite{bertoin_2006}).

More generally, the assumption that $\Set{f \ge \epsilon}$
is finite implies that the coalescent $(\Pi_t)_{t \ge 0}$ obtained from a 
paintbox based on $f$ has only finitely many blocks for any $t > 0$.
This property is usually referred to as ``coming down from infinity''. 
It has been shown in \cite{lambert_random_2017} that any coalescent
which comes down from infinity can be represented as a paintbox
based on a comb, see Proposition~3.2.

    \subsection{General combs} \label{SS:generalComb}

One of the objectives of this work is to extend Theorem~3.1 of
\cite{lambert_comb_2017} and Proposition~3.2 of
\cite{lambert_random_2017} to any ultrametric space (not only compact)
and to any exchangeable coalescent (i.e., beyond the ``coming down from
infinity'' property).  From a technical point of view, we note that this
extension is conceptually harder, and requires the technology of
exchangeable nested compositions which were absent in
\cite{lambert_comb_2017}. This point will be discussed further in
\Cref{SS:composition}.

In order to deal with non-compact metric spaces, we need to generalize
the definition of a comb by relaxing the condition on the finiteness of
$\Set{f \ge \epsilon}$. We will encode combs as functions taking values
in the open subsets of $\OpenInterval{0, 1}$. Any open subset $I$ of
$\OpenInterval{0, 1}$ can be decomposed into an at-most countable union
of disjoint intervals denoted by $(I_i)_{i \ge 1}$. For this reason we
will call an open subset of $\OpenInterval{0, 1}$ an
\emph{interval-partition} and each of the intervals $I_i$ is an
\emph{interval component} of $I$. The space of interval-partitions is
conveniently topologized with the Hausdorff distance on the complement,
$d_H$, defined as
\[
    d_H(I, \tilde{I}) = \sup \Set{d(x, [0,1] \setminus \tilde{I}), x \not\in I} 
        \vee \sup \Set{d(x, [0, 1] \setminus I), x \not\in \tilde{I}}.
\]

We propose to generalize the notion of comb to the notion
of \emph{nested interval-partition}.
\begin{definition}
    A nested interval-partition is a c\`adl\`ag function $(I_t)_{t \ge
    0}$ taking values in the open subsets of $\OpenInterval{0, 1}$
    verifying
    \[
        \forall s \le t,\; I_s \subseteq I_t.
    \]
    Sometimes nested interval-partitions will be called generalized combs
    or even simply combs.
\end{definition}

Let us briefly see how this definition extends the initial comb of
\cite{lambert_comb_2017}. Starting from a comb function $f$, we can
build a nested interval-partition $(I_t)_{t \ge 0}$ as follows
\[
    \forall t > 0,\; I_t = \Set{f < t} \setminus \Set{0, 1}
\]
and 
\[
    I_0 = \interior( \Set{f = 0} )
\]
where $\interior(A)$ denotes the interior of the set $A$. 

Conversely if $(I_t)_{t \ge 0}$ is a nested interval-partition
we can define a comb function $f_I \colon [0, 1] \to \R_+$ as
\[
    f_I(x) = \inf \Set{t \ge 0 \suchthat x \in I_t}.
\]
In general $f_I$ does not fulfill that $\Set{f_I \ge t}$
is finite. A necessary and sufficient condition for this
to hold is that for any $t > 0$, $I_t$ has finitely 
many interval components, and the summation of their
lengths is $1$. If the latter condition is fulfilled, we
say that $I_t$ is proper or equivalently that it
has no dust.

\medskip

A nested interval-partition naturally encodes a (pseudo-)ultrametric
$d_I$ on $[0, 1]$ defined as 
\begin{align*}
    d_I(x, y) &= \inf \Set{t \ge 0 \suchthat
            \text{$x$ and $y$ belong to the same interval of $I_t$}}\\
                                        &= \sup_{[x, y]} f_I
\end{align*}
for $x < y$. We call the ultrametric space $([0, 1], d_I)$ the \emph{comb
metric space} associated to $(I_t)_{t \ge 0}$. 
In order to turn $([0, 1], d_I)$ into a UMS, we need to define an
appropriate $\sigma$-field and a sampling measure. The interval $[0, 1]$
is naturally endowed with the usual Borel $\sigma$-field $\Bi([0, 1])$
and the Lebesgue measure. However, the usual Borel $\sigma$-field does
not fulfill the requirements of \Cref{def:UMS} in general because two
points that belong to the same interval component of $I_0$ are
indistinguishable in the metric $d_I$. This can be addressed by
considering a slightly smaller $\sigma$-field as follows.

Let $(I^0_i)_{i \ge 1}$
be the interval components of $I_0$. We define a $\sigma$-field
$\Is$ on $[0, 1]$ as
\[
    \Is = \Big\{A \cup \bigcup_{i \in M} I^0_i : A \in \Bi([0, 1]
        \setminus I_0) \text{ and } M \subseteq \N\Big\}
\]
where $\Bi([0, 1] \setminus I_0)$ denotes the usual Borel $\sigma$-field
on $[0, 1] \setminus I_0$. It is clear that $\Is \subseteq \Bi([0, 1])$. 
We call \emph{comb metric measure space} associated to $(I_t)_{t
\ge 0}$ the quadruple $([0, 1], d_I, \Is, \Leb)$, where $\Leb$ is the
restriction of the Lebesgue measure to $\Is$. The following lemma shows
that the Lebesgue measure on $\Is$ satisfies the requirements of
\Cref{def:UMS}, and that a comb metric measure space is a UMS.

\begin{lemma} \label{lem:combUMS}
    Any comb metric measure space $([0, 1], d_I, \Is, \Leb)$ is a UMS.
\end{lemma}

\begin{proof}
    Let us first prove that (iii) holds. For $x \in [0, 1]$ and $t
    \ge 0$, let $I_t(x)$ denote the interval component of $I_t$ to
    which $x$ belongs if $x \in I_t$, or let $I_t(x) = \Set{x}$ else.
    Then for $t > 0$ we have
    \[
        \Set{y \in [0, 1] \suchthat d_I(x, y) < t} 
        = \bigcup_{s < t} I_s(x) \in \Is.
    \]
    It remains to show that $\Is \subseteq \Bi_I([0, 1])$, where
    $\Bi_I([0, 1])$ denotes the $\sigma$-field induced by $d_I$. It is 
    sufficient to prove that for all $x, y \not\in I_0$, we have 
    $\OpenInterval{x, y} \in \Bi_I([0, 1])$. Let $z \in \OpenInterval{x, y}$
    and suppose that $z \in I_t$ for all $t > 0$. Then $I_{t_z}(z) \subseteq
    \OpenInterval{x, y}$ for a small enough $t_z$, and thus 
    \[
        \Set{z' \in [0, 1] \suchthat d_I(z, z') < t_z} \subseteq \OpenInterval{x, y}.
    \]
    Otherwise if $z \not\in I_{t_z}(z)$ for some $t_z$, then 
    $\Set{z' \in [0, 1] \suchthat d_I(z, z') < t_z} = 
    \Set{z}$. We can now write
    \[
        \OpenInterval{x, y} = \bigcup_{z \in \OpenInterval{x,y}} 
        \Set{z' \in [0, 1] \suchthat d_I(z, z') < t_z} \in \Bi_I([0, 1])
    \]
    which proves that point (iii) of the definition is fulfilled.

    Let $I_{t-} = \bigcup_{s < t} I_s$, then 
    \[
        \Set{(x,y) \in [0, 1]^2 \suchthat d(x,y) < t} = \Delta_0 \cup
        \bigcup_{x \in I_{t-}} I_t(x) \times I_t(x),
    \]
    where $\Delta_0 = \Set{(x, y) \in ([0, 1] \setminus I_0)^2 \suchthat x = y}$.
    As there are only countably many interval components of $I_t$, the
    union on the right-hand side is countable, and this set belongs to
    the product $\Is \otimes \Is$. This proves that point (i)
    holds and that the comb metric measure space is a UMS.
\end{proof}

For later purpose, let us denote by $U_I$ the completion of the quotient
space of $\Set{f_I = 0}$ by the relation $x \sim y$ \tiff $d_I(x,y) =
0$. (This completion can be realized explicitly by adding countably
many ``left'' and ``right'' faces to the comb, see
\Cref{SS:combCompletion}.)

\medskip

Finally, as in the compact case, an exchangeable coalescent 
$(\Pi_t)_{t \ge 0}$ can be obtained from a nested interval-partition
$(I_t)_{t \ge 0}$ out of an i.i.d.\ uniform sequence $(X_i)_{i \ge 1}$ by
defining
\begin{align} \label{eq:paintbox}
    i \sim_{\Pi_t} j \iff 
        \text{$X_i$ and $X_j$ belong to the same interval component of $I_t$.}
\end{align}
Notice that this definition is a multidimensional extension of the 
original Kingman paintbox procedure, see e.g.\ the beginning of Section~2.3.2
of~\cite{bertoin_2006}.

\begin{remark} \label{rem:cadlag}
    The coalescent obtained through this sampling procedure is not
    c\`adl\`ag in general.  As a coalescent is a non-decreasing process,
    we can (and will) always suppose that we work with a c\`adl\`ag
    modification of the coalescent.
\end{remark}

\begin{remark} \label{rem:samplingComparison}
    We have defined two natural ways of sampling a coalescent from
    a nested interval-partition. First, one can realize the
    extended paintbox procedure described in equation~\eqref{eq:paintbox}.
    Second, one can consider the comb metric measure space associated
    to the nested interval-partition and sample the coalescent according to
    equation~\eqref{eq:UMSsampling}. It is not hard to see that the
    coalescent obtained through~\eqref{eq:UMSsampling} is the c\`adl\`ag
    version of the one obtained through~\eqref{eq:paintbox}.
\end{remark}

We will now demonstrate that nested interval-partitions form a large
enough framework to answer our two initial problems: representing any
exchangeable coalescent as a paintbox on a comb and representing general
ultrametric measure spaces.

    \subsection{Comb representation of exchangeable coalescents}
    \label{SS:introCoalescent}

\paragraph{General comb representation} 
We start by showing that one can always find a comb representation of any
coalescent. First notice that this representation cannot be unique.  For
example taking the reflection of a comb about the vertical line in the
middle of the segment $[0, 1]$ yields a new comb but does not change the
associated coalescent.  In many applications we will not be interested in
this order but only in the genealogical structure of the comb. For this
reason we introduce the following relation.

\begin{definition} \label{def:paintboxEquivalent}
    Two generalized combs are paintbox-equivalent if 
    their associated coalescents
    are identical in law.
    Being paintbox-equivalent is an equivalence relation,
    we denote by $\mathfrak{I}$ the quotient space. 
\end{definition}

Given $I \in \mathfrak{I}$ we denote by $\rho_I$ the distribution on the space
of coalescents of the paintbox based on any representative of $\If$. We
provide the following version of Kingman's representation theorem (e.g.\
see~\cite{bertoin_2006} Theorem~2.1) for exchangeable coalescents.

\begin{theorem} \label{TH:coalescentGeneral}
    Let $(\Pi_t)_{t \ge 0}$ be an exchangeable coalescent. There exists 
    a unique distribution $\nu$ on $\If$ such that 
    \[
        \Prob*{\big}{(\Pi_t)_{t \ge 0} \in \cdot} 
        = \int_{\If} \rho_I(\cdot) \nu(\mathrm{d}I).
    \]
\end{theorem}

\begin{remark} 
    It is interesting to relate this result to the original theorem from
    Kingman. A \emph{mass-partition} is a sequence 
    $\beta = (\beta_i)_{i \ge 1}$ such that
    \[
        \beta_1 \ge \beta_2 \ge \dots \ge 0, \quad \sum_{i \ge 1} \beta_i \le 1.
    \]
    Kingman's representation theorem states that any exchangeable 
    partition can be obtained through a paintbox based on a random
    mass-partition, and that this correspondence is bijective. 
    A mass-partition can be seen as the ranked sequence of the lengths
    of the interval components of an interval-partition. Now notice 
    that two interval-partitions are paintbox-equivalent, i.e.\ induce the 
    same exchangeable partition, \tiff they
    have the same associated mass-partition. In this one-dimensional
    setting, any paintbox-equivalence class of interval-partitions can
    be identified with a random mass-partition. In a similar way, it would be
    natural to try to identify the elements of $\If$ 
    with mass-partition valued processes, also called \emph{mass-coalescents}.
    However, one can easily find two different equivalence classes
    of $\If$ that have the same associated mass-coalescent, see \Cref{F:counter}. 
\end{remark}

\begin{remark}
    A result very similar to \Cref{TH:coalescentGeneral} has been
    obtained in~\cite{Forman2017}, Theorem~4, in the context of
    hierarchies. Roughly speaking, an exchangeable hierarchy is obtained
    from an exchangeable coalescent by ``forgetting about time''. In this
    sense, an exchangeable coalescent carries more information, and this
    part of our work can be seen as an extension of~\cite{Forman2017}.
    However, the forthcoming \Cref{S:lambda} and \Cref{S:ultrametric}
    heavily rely on the knowledge of the coalescence times, and could not
    have been achieved in the framework of hierarchies. We have dedicated
    \Cref{app:hierarchies} to the explanation of the links between the
    present work and~\cite{Forman2017}.
\end{remark}

\begin{figure}
    \center
    \includegraphics[width=\textwidth]{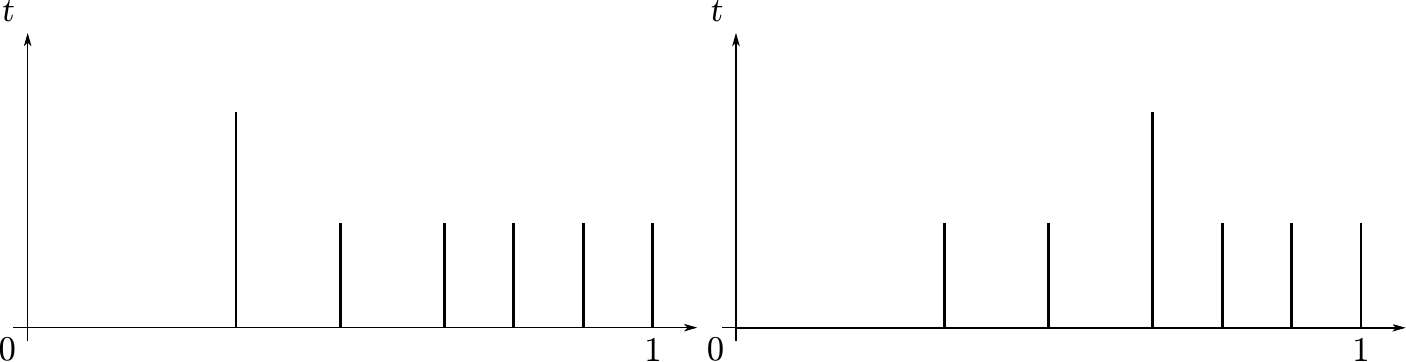}
    \caption{An example of two nested interval-partitions that have the same
    mass-coalescent but different coalescents. For both processes,
    the initial mass-partition is $(\frac{1}{3}, \frac{1}{6},
    \frac{1}{6}, \frac{1}{9}, \frac{1}{9}, \frac{1}{9}, 0, \dots)$,
    then $(\frac{2}{3}, \frac{1}{3}, 0, \dots)$ and finally $(1, 0, \dots)$. However, for
    the process on the left-hand side the first blocks to merge are those of mass $1/6$ and 
    $1/9$, whereas for the right-hand process, the
    blocks of mass $1/6$ first merge with the block of size $1/3$.}
    \label{F:counter}
\end{figure}

\paragraph{$\Lambda$-coalescents}
Most of the efforts made in the study of exchangeable coalescents have
been devoted to the special case of $\Lambda$-coalescents
\citep{pitman_coalescents_1999, sagitov1999}.
These coalescents are parametrized by a finite measure 
$\Lambda$ on $[0, 1]$, and their restriction to $[n] \defas \Set{1, \dots, n}$
is a Markov chain whose transitions are the following.
The process undergoes a transition from a partition $\pi$ with
$b$ blocks to a partition obtained by merging $k$ blocks
of $\pi$ at rate $\lambda_{b, k}$ given by
\[
    \lambda_{b, k} = \int_{[0,1]} x^{k-2} (1-x)^{b-k} \Lambda(\mathrm{d}x).
\]
The next proposition states that we can always find a Markovian
comb representation of a $\Lambda$-coalescent. Moreover
in \Cref{S:lambda} we provide an explicit description of its transition.

\begin{proposition} \label{TH:coalescentLambda}
    Let $(\Pi_t)_{t \ge 0}$ be a $\Lambda$-coalescent. There exists
    $(I_t)_{t \ge 0}$ a Markov nested interval-partition such that
    the coalescent obtained from the paintbox based on $(I_t)_{t \ge 0}$
    is distributed as $(\Pi_t)_{t \ge 0}$. 
\end{proposition}

\begin{remark}[Combs and the flow of bridges]
    The flow of bridges introduced by~\cite{bertoin_stochastic_2003}
    represents the dynamics of a population whose genealogy is given by a
    $\Lambda$-coalescent. We will show that we can build a nested
    interval-partition from the flow of bridges and that it has the same
    distribution as the Markov nested interval-partition of
    \Cref{TH:coalescentLambda}, see \Cref{S:lambda}.
\end{remark}

\begin{remark}
    There exists a natural extension of the $\Lambda$-coalescents called
    the coalescents with simultaneous multiple collisions or
    $\Xi$-coalescents \citep{schweinsberg_coalescents_2000}. All our
    results carry over to $\Xi$-coalescents, however for the sake of
    clarity we will focus on the case of $\Lambda$-coalescents.
\end{remark}

A coalescent process models the genealogy of a population living at a
fixed observation time. Many works have been concerned with the dynamical
genealogy obtain by varying the observation time of the population. For
example, in~\cite{pfaffelhuber_process_2006, pfaffelhuber_tree_2011} the
authors study some statistics of the dynamical genealogy, namely the time
to the MRCA and the total length of the genealogy.
In~\cite{greven_tree-valued_2008} the genealogy is encoded as a metric
space (a real tree, see~\cite{evans_probability_2007}) and the authors
introduce the tree-valued Fleming-Viot process, a process bearing the
entire information on the dynamical genealogy. This encoding requires to
work with metric space-valued stochastic processes, and with the rather
technical Gromov-weak topology for metric spaces. 

We address such questions in the framework of combs in
\Cref{SS:dynamicalComb}. We show that we can naturally encode a
dynamical genealogy as a comb-valued process called the \emph{evolving
comb}. This process is a Markov process, whose semi-group can be
explicitly described. In the particular case of coalescents that come
down from infinity, the semi-group of the evolving comb takes a
particularly simple form in terms of sampling from an independent comb.

    \subsection{Comb representation of ultrametric spaces}
    \label{SS:introUltrametric}

The second main aim of this paper is to provide a comb representation of
ultrametric measure spaces in the same vein as Theorem~3.1
of~\cite{lambert_comb_2017}.  We will only state our results informally
and refer to \Cref{S:ultrametric} for the precise statements. 

We first introduce the \emph{Gromov-weak topology} on the space of UMS and
show that any UMS is indistinguishable from a comb metric space in this
topology. To do so, we realize a straightforward extension of the work
developed in~\cite{greven_convergence_2006, gromov2007metric} which is
focused on separable metric measure spaces. In short, starting from a UMS a
we can obtain a coalescent by sampling from it as described in
\Cref{SS:intro}. We say that a sequence of UMS converges to a limiting
UMS in the Gromov-weak sense if the corresponding coalescents converge
weakly as partition-valued stochastic processes (see \Cref{SS:gromovWeak}
for a more precise definition). We are now ready to state our
representation result, which is a direct application of
\Cref{TH:coalescentGeneral}.

\begin{theorem} \label{TH:combRepresentation}
    For any UMS $(U, d, \Ui, \mu)$ there exists a comb metric measure
    space that is indistinguishable in the Gromov-weak topology from $(U,
    d, \Ui, \mu)$. 
\end{theorem}

\begin{proof}
    As we have identified any UMS with the distribution of its
    coalescent, two UMS are indistinguishable \tiff their coalescents
    have the same distribution. \Cref{TH:coalescentGeneral} shows
    that we can always find a nested interval-partition $(I_t)_{t \ge
    0}$ such that the coalescent obtained from a paintbox based on
    $(I_t)_{t \ge 0}$ is distributed as the coalescent obtained by
    sampling from $(U, d, \Ui, \mu)$. As noticed in
    \Cref{rem:samplingComparison}, the coalescent obtained by
    sampling in the comb metric measure space $([0, 1], d_I, \Is, \Leb)$
    has the same distribution as the coalescent obtained from the
    paintbox based on $(I_t)_{t \ge 0}$, and thus this comb metric
    measure space is indistinguishable from $(U, d, \Ui, \mu)$.
\end{proof}

The comb representation given by \Cref{TH:combRepresentation} is rather
weak, since it only ensures that we can find a comb that has the same
sampling structure as a given UMS. We would like to be more precise and
obtain an isometry result as in the compact case.  This is not possible
in general, and we have to consider separately the separable case and the
non-separable case.

\paragraph{The separable case} 
In the separable case, the coalescent contains all the information about
the UMS. More precisely, the Gromov reconstruction theorem ensures that
two complete separable UMS that are indistinguishable in the Gromov-weak
topology have the supports of their measures in isometry, see e.g.\
\cite{gromov2007metric}, Section~3.$\frac{1}{2}$.5
or~\cite{greven_convergence_2006}, Proposition~2.6. 
The following refinement of \Cref{TH:combRepresentation} in the separable
case is a direct consequence of the Gromov reconstruction theorem and of 
\Cref{TH:combRepresentation}, see \Cref{SS:separable} for a proof.

\begin{corollary} \label{cor:isometrySeparableUMS}
    Let $(U, d, \Ui, \mu)$ be a complete separable UMS. There exists
    a comb metric measure space such that the support of $\mu$ is
    isometric to $(U_I, d_I)$, and that the isometry maps $\mu$ to
    $\Leb$.
\end{corollary}

Additionally, any separable ultrametric space $(U, d)$
can be endowed with a probability measure whose support is the whole
space $U$, see \Cref{lem:measure}.  This result combined with
\Cref{cor:isometrySeparableUMS} yields the following representation
result for complete separable ultrametric spaces, which is the direct extension of
Theorem~3.1 of \cite{lambert_comb_2017} to the separable case.

\begin{proposition} \label{prop:representationSeparable}
    Let $(U, d)$ be a complete separable ultrametric space. We can find a
    nested interval-partition such that $(U_I, d_I)$ is isometric to $(U,
    d)$. 
\end{proposition}

A proof of this proposition is provided in \Cref{SS:separable}.
Notice that the proof of the previous proposition is very different from
the original proof of~\cite{lambert_comb_2017} which is no longer valid
for non-compact UMS.

\paragraph{The general case} 
In general, two UMS that are associated to the same coalescent are not
isometric. This essentially comes from the fact that a coalescent only
bears the information about a sequence of ``typical'' points of the UMS,
and that a non-separable UMS may contain more information than the
topology generated by these ``typical'' points. The main idea of our
approach relies on a new decomposition that we now expose.

An UMS $(U, d, \Ui, \mu)$ can be seen as the leaves of a tree. We show
that we can decompose this tree into two parts. The first part is a
separable tree that we call the backbone. Secondly, one can then recover
the tree from the backbone by grafting some ``simple'' subtrees on the
backbone.  By ``simple'', we mean that each of those subtrees has the
sampling properties of a star-tree, in the sense that all points
sampled in the same subtree are at the same distance to each other. See
\Cref{F:quotient} for an illustration of this decomposition, and
\Cref{def:backbone} for a precise definition of the backbone. An object
very similar to the backbone is studied
in~\cite{gufler_representation_2016} but the construction of the backbone
from a general UMS is not considered there.

Our result states that if two UMS have complete backbones and are associated
to the same coalescent, then the backbones are in isometry in a way
that preserves the star-trees attached to it. We say that the two UMS
are in \emph{weak isometry}, see \Cref{def:isometry-backbone}. We provide
the following version of the Gromov reconstruction theorem in the
case of general UMS.

\begin{proposition} \label{prop:weakIsometry}
    Let $(U, d, \Ui, \mu)$ and $(U', d', \Ui', \mu')$ be two UMS with
    complete backbones. These UMS are indistinguishable in the
    Gromov-weak topology \tiff $(U, d, \Ui, \mu)$ and $(U', d',
    \Ui', \mu')$ are in weak isometry.
\end{proposition}

An equivalent reformulation of the previous proposition is stated in
\Cref{SS:samplingEquivalence}, see \Cref{prop:samplingEquivalence},
and proved at the end of \Cref{SS:samplingEquivalence}. As a
consequence of \Cref{prop:weakIsometry} and
\Cref{TH:combRepresentation}, we have the following
version of Theorem~3.1 of~\cite{lambert_comb_2017} in the general
case. See \Cref{SS:combCompletion} for a proof.

\begin{corollary} \label{cor:representationGeneral}
    Let $(U, d, \Ui, \mu)$ be a UMS with a complete backbone. There exists
    a nested interval-partition $(I_t)_{t \ge 0}$ such that, up to the
    addition of a countable number of points, the comb metric measure
    space $([0, 1], d_I, \Is, \Leb)$ is weakly isometric to $(U, d, \Ui, \mu)$.
\end{corollary}

\begin{figure}
    \center
    \includegraphics[width=.8\textwidth]{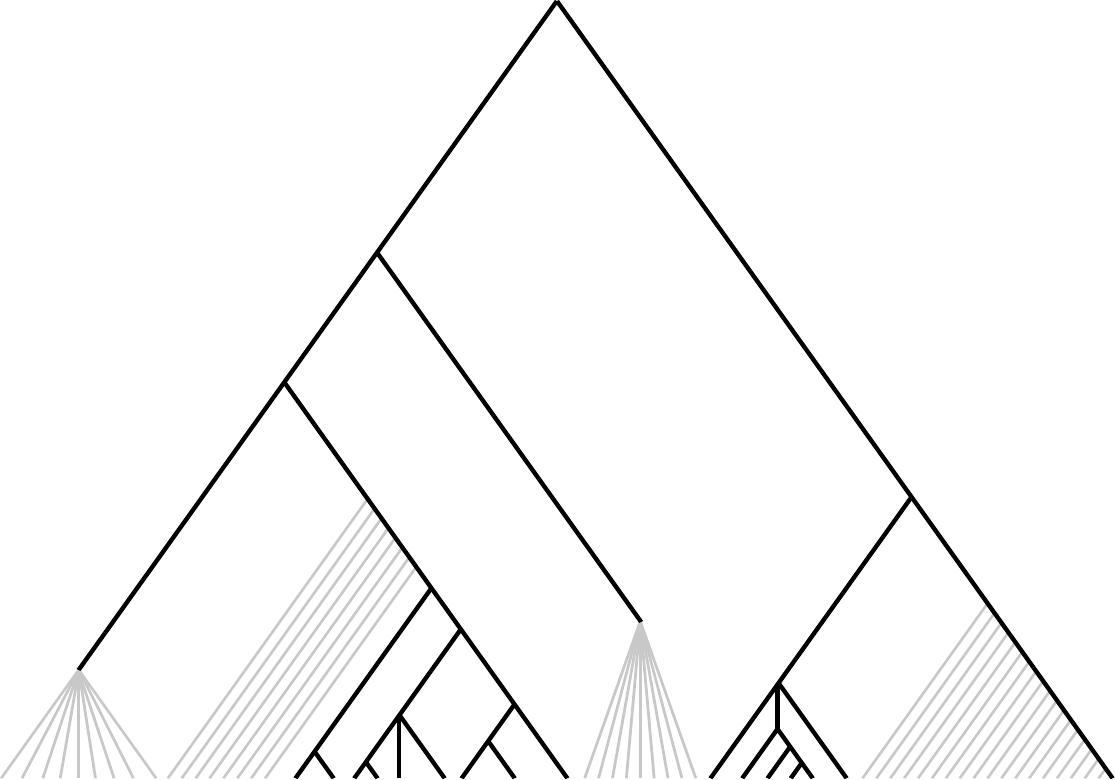}
    \caption{Illustration of the backbone decomposition. The dark thick
    lines represent the backbone. An element of the tree is represented
    in grey if its descendance has zero mass.}
    \label{F:quotient}
\end{figure}

        \subsection{Outline} 

The rest of the paper is divided into three parts. 

In \Cref{S:composition} we introduce the notion of composition and nested
composition which will be our main tool to study combs.
\Cref{SS:composition} introduces the existing material on random
compositions. In \Cref{SS:nestedComposition} we define exchangeable
nested compositions and prove the representation theorem linking combs
and nested compositions. The proof of \Cref{TH:coalescentGeneral} is
given in \Cref{SS:coalescent}.

In \Cref{S:lambda} we restrict our attention to the case of
$\Lambda$-coalescents.  We define there the notion of a $\Lambda$-comb
and study a family of nested compositions emerging from the
$\Lambda$-coalescents. The proof of \Cref{TH:coalescentLambda} is given
in \Cref{SS:proofSemigroup}.  The evolving comb is introduced and studied
in \Cref{SS:dynamicalComb}.

Finally in \Cref{S:ultrametric} we envision combs as ultrametric spaces.
A precise outline of this section is given at the beginning of
\Cref{S:ultrametric}.

\section{Combs and nested compositions}
\label{S:composition}

The objective of this section is to prove \Cref{TH:coalescentGeneral} 
on the comb representation of exchangeable coalescents. As was 
already mentioned in introduction, the correspondence between 
combs and exchangeable coalescents cannot be bijective. Roughly
speaking, this comes from the fact that a nested interval-partition
inherits an order from $[0, 1]$, and that changing this order
does not modify the associated coalescent. However, we will
show in \Cref{SS:nestedComposition} that there is a bijective correspondence
between nested interval-partitions and exchangeable nested compositions,
the ordered version of exchangeable coalescents. Exchangeable
nested compositions will be our main tool to study combs. 

We start this section by recalling existing results and material on exchangeable
compositions developed in~\cite{gnedin_representation_1997, donnelly_consistent_1991}
and then show how to extend them to nested compositions.

    \subsection{Exchangeable compositions}\label{SS:composition}

In combinatorics, a composition of $[n]$ (resp.\ $\N$) is a 
partition of $[n]$ (resp.\ $\N$) with a total order on the blocks.
We write $\Ci = (\pi, \le)$ for a composition of $\N$ where $\pi$ is
the partition and $\le$ the order on the blocks. 
The blocks of the partition $\pi$ can always be labeled in
increasing order of their least element, i.e.\ the blocks
of $\pi$ are denoted by $(A_1, A_2, \dots)$ and are such 
that for any $i,j \ge 1$,
\[
    i \le j \iff \min(A_i) \le \min(A_j).
\]
Let $\sigma$ be a finite permutation of $\N$, we denote by
$\sigma(\Ci)$ the composition whose blocks are $(\sigma(A_1), \sigma(A_2), \dots)$
and such that the order of the blocks is
\[
    \sigma(A_i) \le \sigma(A_j) \iff A_i \le A_j.
\]
For example, for $n = 5$, consider $\Ci^n$ the composition
\[
    \Ci^n = \Set{2, 3} \le \Set{5} \le \Set{1, 4}.
\]
With our labeling convention, we have 
$A_1 = \Set{1, 4}$, $A_2 = \Set{2, 3}$ and $A_3 = \Set{5}$ ($A_1$ needs not be
the first block of $\Ci$ for the order $\le$).
If $\sigma = (2, 1, 3, 5, 4)$, the composition $\sigma(\Ci^n)$ is given by
\[
    \sigma(\Ci^n) = \Set{1, 3} \le \Set{4} \le \Set{2, 5}.
\]  
A random composition $\Ci$ of $\N$ is called \emph{exchangeable} if for any finite
permutation $\sigma$,
\[
    \Ci \overset{\text{(d)}}{=} \sigma(\Ci).
\]

\cite{gnedin_representation_1997} provides a procedure to build an
exchangeable composition of $\N$ from any interval-partition $I$ called
the \emph{ordered paintbox}. Let $(V_i)_{i \ge 1}$ be an i.i.d.\ sequence
of uniform $[0,1]$ variables.  Let $\Ci$ be the composition of $\N$ whose
blocks are given by the relation
\[
    i \sim j \iff \text{$V_i$ and $V_j$ belong to the same interval component
    of $I$}
\]
and the order of the blocks is
\[
    A \le A' \iff V_i \le V_j, \quad \forall i \in A,\; \forall j \in A'.
\]
The main result of \cite{gnedin_representation_1997} shows that any
exchangeable composition of $\N$ can be obtained as an ordered paintbox
based on a random interval-partition (see Theorem~11 in
\cite{gnedin_representation_1997}). We now give a proof of this
result that differs from the original proof
of~\cite{gnedin_representation_1997}. We make use of de Finetti's theorem
in a similar way as Aldous' proof of Kingman's theorem, see e.g.\ the 
proof of Theorem~2.1 in~\cite{bertoin_2006}. The original proof
of~\cite{gnedin_representation_1997} relies on a reversed martingale
argument combined with the method of moments.

\begin{theorem}[Gnedin] \label{TH:gnedin}
    Let $\Ci$ be an exchangeable composition of $\N$. There exists
    on the same probability space a random
    interval-partition $I$ and an independent i.i.d.\ sequence $(V_i)_{i
    \ge 1}$ of uniform $[0, 1]$ variables such that the ordered paintbox
    based on $I$ by the sequence $(V_i)_{i \ge 1}$ is a.s.\ $\Ci$.
\end{theorem}

Before showing the theorem we need a technical lemma. Any composition
$\Ci = (\pi, \le)$ can be encoded as a total preorder $\preceq$ on
$\N$ defined as
\[
    i \preceq j \iff B_i \le B_j
\]
where $B_i$ (resp.\ $B_j$) is the block containing $i$ (resp.\ $j$). 
The blocks of $\pi$ can be recovered from $\preceq$ by the following 
relation
\[
    i \sim j \iff i \preceq j \text{ and } j \preceq i
\]
and the order $\le$ by 
\[
    B \le B' \iff i \preceq j, \quad \forall i \in B, \forall j \in B'.
\]
\begin{lemma}
    Let $\Ci$ be an exchangeable composition of $\N$. We can find
    an exchangeable sequence of $[0, 1]$-valued random variables $(\xi_i)_{i \ge 1}$
    such that
    \[
        i \preceq j \iff \xi_i \le \xi_j.
    \]
\end{lemma}
\begin{proof}
    Let $D_i$ be the set of integers lower than $i$
    \[
        D_i = \Set{k \suchthat k \preceq i}.
    \]
    It is immediate that the partition $(D_i \setminus \Set{i}, \N
    \setminus \Set{i} \setminus D_i)$ 
    is an exchangeable partition of $\N \setminus \Set{i}$. Thus Kingman's
    representation theorem (see e.g.\ Theorem~2.1 in \cite{bertoin_2006})
    ensures that the limit
    \[
        \xi_i = \lim_{n \to \infty} \frac{1}{n} \Card(D_i \cap [n])
    \]
    exists a.s. Fix a finite permutation $\sigma$ whose support lies in
    $[n]$, i.e.\ such that $\sigma(i) = i$ for $i \ge n$. For $m \ge n$, 
    the distribution of $(\Card(D_i \cap [m]))_{i \ge 1}$ is invariant by
    the action of $\sigma$. Taking the limit, the distribution of the
    sequence $(\xi_i)_{i \ge 1}$ is also invariant by the action of
    $\sigma$, and thus it is an exchangeable sequence.

    We need to show that 
    \[
        i \preceq j \iff \xi_i \le \xi_j.
    \]
    The only difficulty here is to show that $\xi_i \le \xi_j$ implies $i
    \preceq j$. Suppose that $i \not\preceq j$, we need to show that 
    \[
        \xi_i - \xi_j = \lim_{n \to \infty} \frac{1}{n}\Card\big( (D_i \setminus D_j) \cap [n] \big) > 0.
    \]
    The partition $(D_j \setminus \Set{i,j},\; D_i \setminus \Set{i,j} \setminus D_j,\; 
    \N \setminus \Set{i,j} \setminus D_i)$ is an exchangeable partition of 
    $\N \setminus \Set{i,j}$. Another interesting consequence of Kingman's theorem
    is that in any exchangeable partition, the blocks are either singletons or have 
    positive asymptotic frequencies. According to this, it is sufficient to show that 
    a.s.\ $D_i \setminus D_j$ has at least two elements that are not $i$. 
    Consider $B_i$ (resp.\ $B_j$)
    the block to which $i$ (resp.\ $j$) belongs. The set $D_i \setminus D_j$
    is the reunion of all the blocks $B$ such that $B_j < B \le B_i$. 
    Thus $D_i \setminus D_j$ is a singleton \tiff $B_i = \Set{i}$
    and there exists at most one singleton block $B$ such that $B_j < B < B_i$. Let $n \ge 1$
    and consider the block sizes and order of $\Ci^n$ as fixed. Exchangeability
    shows that the labels inside the blocks are chosen uniformly among all
    the possibilities. In particular this shows that the probability that 
    $(D_i \setminus D_j) \cap [n]$ is a singleton goes to $0$ as $n$ goes
    to infinity.
\end{proof}

Now \Cref{TH:gnedin} is essentially a corollary of the previous lemma and of de Finetti's theorem.
\begin{proof}[Proof of \Cref{TH:gnedin}]
    Let $(\xi_i)_{i \ge 1}$ be as above. Applying de Finetti's theorem
    we know that there exists a random measure $\mu$ such that conditionally on it the
    sequence $(\xi_i)_{i \ge 1}$ is i.i.d.\ distributed as $\mu$.
    Consider the distribution function $F_\mu$ of $\mu$, and its generalized
    inverse
    \[
        F_\mu^{-1}(x) = \inf \Set{r \suchthat F_\mu(r) > x}.
    \]
    The interval-partition associated with $\mu$, $I_\mu$, is defined as the set of flats of
    $F^{-1}_\mu$:
    \[
        I_\mu = \Set{x \in [0,1] \suchthat \exists y < x < z,\; F^{-1}_\mu(y) = F^{-1}_\mu(z)}.
    \]

    The measure $\mu$ has the property that if $X$ is distributed as $\mu$, then 
    $\mu$-a.s.\ $F_\mu(X) = X$. Conditioning on $\mu$, this can be seen from the
    definition of the sequence $(\xi_i)_{i \ge 1}$ and the law of large numbers:
    \[
        F_\mu(\xi_1) = \lim_{n \to \infty} \frac{1}{n} \sum_{j = 1}^n \Indic{\xi_j \le \xi_1}
                   = \lim_{n \to \infty} \frac{1}{n} \sum_{j = 1}^n \Indic{j \preceq 1}
                   = \xi_1 \quad \text{$\mu$-a.s.}
    \]
    In the terminology of \cite{gnedin_representation_1997} this shows
    that the measure $\mu$ is \emph{uniformized}.  A uniformized measure
    has an atomic and a diffuse part. The support of the diffuse part is
    $[0, 1] \setminus I_\mu$ and coincides with the Lebesgue measure. The
    atomic part is supported by the right endpoints of the interval
    components of $I_\mu$.  If $J = \OpenInterval{\ell, r}$ is an interval
    component of $I_\mu$, the measure $\mu$ has an atom of mass $r-\ell$
    located at $r$.

    Let $(J_k)_{k \ge 1}$ be the interval decomposition of $I_\mu$, and write
    $J_k = \OpenInterval{\ell_k, r_k}$. Let $(X_i)_{i \ge 1}$ be an independent
    i.i.d.\ sequence of uniform variables, we define
    \[
        V_i =
        \begin{cases}
            \xi_i &\text{if $\xi_i \not \in I_\mu$} \\
            (r_k - \ell_k) X_i + \ell_k &\text{if $\xi_i = r_k$}.
        \end{cases}
    \]
    In words, the variables from the sequence $(\xi_i)_{i \ge 1}$ which
    are equal to the atom $r_k$ are uniformly dispersed over the interval
    $J_k$.  The previous remarks on the structure of uniformized measures
    show that conditionally on $\mu$, the sequence $(V_i)_{i \ge 1}$ is
    i.i.d.\ uniform on $[0, 1]$.  The conditional distribution does not
    depend on $\mu$, thus the sequence $(V_i)_{i \ge 1}$ is independent
    of $\mu$ and of $I_\mu$.

    We only need to show that the ordered paintbox based on $I_\mu$ using the sequence
    $(V_i)_{i \ge 1}$ is $\Ci$ a.s. This is plain from the design of the sequence. 
\end{proof}

We end this section with a technical result already present in
\cite{gnedin_representation_1997} (see Proposition~9) which we will
require. Let $\Ci$ be an exchangeable composition of $\N$ and $\Ci^n$ its
restriction to $[n]$.  Let us denote by $n_i$ the size of the $i$-th
block of $\Ci^n$. The \emph{empirical interval-partition} associated to
$\Ci_n$ is given by
\[
    I^n = (0, \frac{n_1}{n}) \cup (\frac{n_1}{n}, \frac{n_1 + n_2}{n}) \cup
    \dots \cup (\frac{n_1 + \dots + n_{k-1}}{n}, 1).
\]
Here is a more pictorial way of constructing $I^n$. Divide $[0, 1]$ in
intervals of size $1/n$ and label them from $1$ to $n$ in such a way that
$i \preceq j$ \tiff the block with label $i$ is before the block with
label $j$. Then $I^n$ is obtained by merging the intervals whose labels
are in the same block of the composition.  The next result states that
the interval-partition representing $\Ci$ in \Cref{TH:gnedin} can be
obtained as the limit of the empirical interval-partitions.
\begin{proposition} \label{PR:empirical}
    If $\Ci$ is an exchangeable composition of $\N$, $I$ the interval-partition
    obtained from \Cref{TH:gnedin} and $(I^n)_{n \ge 1}$ the sequence of empirical
    interval-partitions associated to $\Ci$, we have
    \[
        \lim_{n \to \infty} d_H(I^n, I \setminus \Set{0, 1}) = 0 \quad a.s.
    \]
\end{proposition}

\begin{proof}
    Let $\mu$, $(\xi_i)_{i \ge 1}$ and $I_\mu$ be as in the proof of \Cref{TH:gnedin}. 
    De Finetti's theorem ensures that
    \[
        \lim_{n \to \infty} \mu_n \defas \frac{1}{n} \sum_{i = 1}^n \delta_{\xi_i} = \mu \quad a.s.
    \]
    in the sense of weak convergence of probability measures. The
    interval-partition $I_{\mu_n}$ coincides with the empirical
    interval-partition $I^n$ and as was already noticed in
    \cite{gnedin_representation_1997}, the weak convergence of $\mu_n$ to
    $\mu$ implies the convergence of $I_{\mu_n}$ to $I$ in the Hausdorff
    topology.
\end{proof}

\begin{remark}
    This also shows that the representation obtained through
    \Cref{TH:gnedin} is unique in distribution. The interval-partition
    $I$ is a.s.\ recovered from $I^n$ whose distribution is fully
    determined by $\Ci$. 
\end{remark}

    \subsection{Exchangeable nested compositions}
    \label{SS:nestedComposition}

Gnedin's theorem sets up a correspondence between random
interval-partitions and exchangeable compositions. We want to find a
similar correspondence between nested interval-partitions and
exchangeable nested compositions, the ordered version of exchangeable
coalescents.  A nested composition of $[n]$ (resp.\ $\N$) is a
c\`adl\`ag process $(\Ci_t)_{t \ge 0}$ taking values in the
compositions of $[n]$ (resp.\ $\N)$ such that, as $t$ increases, only
adjacent blocks of the composition merge. More precisely, if $(\Ci_t)_{t
\ge 0}$ is a nested composition, for any $s \le t$, the blocks of $\Ci_t$
are obtained by merging blocks of $\Ci_s$, and if $A \le B$ are two
blocks of $\Ci_s$ that merge, they also merge with any block $C$ such
that $A \le C \le B$.

Naturally we say that $(\Ci_t)_{t \ge 0}$ is an exchangeable nested composition
of $\N$ if for any finite permutation $\sigma$ we have
\[
    (\Ci_t)_{t \ge 0} \overset{\text{(d)}}{=} (\sigma(\Ci_t))_{t \ge 0}.
\]
We can extend the ordered paintbox construction to nested compositions. Let
$(I_t)_{t \ge 0}$ be a nested interval-partition, and $(V_i)_{i \ge 1}$
an independent i.i.d.\ uniform sequence. Let $\Ci_t$ be the composition
obtained from the ordered paintbox based on $I_t$ by $(V_i)_{i \ge 1}$. 
Then it is immediate that $(\Ci_t)_{t \ge 0}$ is an exchangeable nested
composition. Notice that this is only true because we have used the same 
sequence $(V_i)_{i \ge 1}$ for all times $t$. 

\begin{remark} \label{rem:compositionCadlag}
    Similarly to \Cref{rem:cadlag}, the nested composition obtained
    from an ordered paintbox is not c\`adl\`ag in general. Again it
    admits a unique c\`adl\`ag modification and we shall always consider this
    modification.
\end{remark}

We have the following direct reformulation of \Cref{TH:gnedin} in the
framework of nested compositions.

\begin{theorem} \label{TH:gnedinNested}
    Let $(\Ci_t)_{t \ge 0}$ be an exchangeable nested composition of
    $\N$.  We can find on the same probability space a nested
    interval-partition $(I_t)_{t \ge 0}$ and an independent i.i.d.\
    sequence $(V_i)_{i \ge 1}$ of uniform variables such that a.s.\ the
    ordered paintbox based on $(I_t)_{t \ge 0}$ with $(V_i)_{i \ge 1}$ is
    $(\Ci_t)_{t \ge 0}$. This nested interval-partition is unique in
    distribution.
\end{theorem}

\begin{proof}
    \textit{Existence.}
    For any $t \ge 0$, $\Ci_t$ is an exchangeable composition of $\N$. We can
    apply \Cref{TH:gnedin} distinctly for 
    $t \in \Q_+$ to find on the same probability space a collection of 
    interval-partitions $(I_t)_{t \in \Q_+}$ such that for any $t \in \Q_+$
    the ordered paintbox based on $I_t$ is $\Ci_t$. Let $I^n_t$ be the
    empirical interval-partition associated to $\Ci_t \cap [n]$. The fact
    that $(\Ci_t)_{t \ge 0}$ is a nested composition ensures that 
    $(I^n_t)_{t \in \Q_+}$ is a nested interval-partition. Taking the 
    limit as $n$ goes to infinity shows that $(I_t)_{t \in \Q_+}$ is also
    a nested interval-partition. It admits a unique c\`adl\`ag extension given by
    \[
        I_s = \interior(\bigcap_{\substack{t \ge s\\ t \in \Q_+}} I_t).
    \]

    Let $(V_i)_{i \ge 1}$ be the i.i.d.\ uniform sequence given by 
    \Cref{TH:gnedin} applied at time $t = 0$.
    To see that $(V_i)_{i \ge 1}$ is independent of $(I_t)_{t \ge 0}$,
    one can do the exact same steps as in the proof of \Cref{TH:gnedin} 
    but using a vectorial version of de Finetti's theorem (see
    \Cref{app:detailsFinetti}).

    We now show that for any $t \in \Q_+$, a.s.
    \begin{align} \label{eq:equivalenceSampling}
        i \sim_t j \iff \text{$V_i$ and $V_j$ are in the same interval of $I_t$}
    \end{align}
    where $\sim_t$ is the relation given by the blocks of $\Ci_t$. 
    
    Let $n \ge 1$ and divide the interval $[0, 1]$ in $n$ intervals of size
    $1/n$. We label the intervals from $1$ to $n$ in the same order as 
    the variables $V_1, \dots, V_n$. Let $t \in \Q_+$, the first step
    is to notice that the empirical interval-partition $I^n_t$ can be
    recovered by merging the blocks of size $1/n$ whose labels belong 
    to the same block of $\Ci_t$. Now, let $V^{(n)}_i$ (resp.\ $V^{(n)}_j$)
    be the right-hand extremity of the interval with label $i$
    (resp.\ $j$). Using twice the law of large numbers 
    shows that $V^{(n)}_i$ and $V^{(n)}_j$ converge to $V_i$ and $V_j$
    respectively. Moreover, we know that $I^n_t$ converges a.s.\ to $I_t$. 
    If we suppose that $V_i < V_j$ and $i \sim_t j$, then for any $n \ge 1$,
    $(V^{(n)}_i, V^{(n)}_j) \subset I^n_t$, and taking the limit shows that
    $(V_i, V_j) \subset I_t$. Conversely if $(V_i, V_j) \subset I_t$, using
    the convergence, for $n$ large enough we have $(V^{(n)}_i, V^{(n)}_j) \subset I^n_t$
    and thus $i$ and $j$ are in the same block of $\Ci_t$. 

    That relation~\eqref{eq:equivalenceSampling} holds a.s.\ for any $t \ge 0$
    will follow by right-continuity. However 
    we have to be careful, in general the nested composition obtained
    from an ordered paintbox is not c\`adl\`ag. By continuity, the relation~\eqref{eq:equivalenceSampling}
    only holds a.s.\ for all times $t$ when $(\Ci_t)_{t \ge 0}$ is continuous. 
    The original nested composition $(\Ci_t)_{t \ge 0}$ is recovered by considering
    a c\`adl\`ag modification of the nested composition obtained though an 
    ordered paintbox based on $(I_t)_{t \ge 0}$. 
    
    \medskip

    \textit{Uniqueness.} The uniqueness will come from the following
    convergence result
    \[
        \lim_{n \to \infty} \sup_{t \ge 0} d_H(I^n_t, I_t) = 0 \quad a.s.
    \]
    We start by showing the convergence. 
    Let $\epsilon > 0$, we can split $[0,1]$ into a finite number of pairwise disjoint intervals
    of length smaller than $\epsilon$ denoted by $J_1, \dots, J_p$. Given
    a combination of such intervals, $J = J_{i_1} \cup \dots \cup J_{i_k}$,
    let $f^n_J$ denote the fraction of variables $V_1, \dots, V_n$ which 
    belong to $J$. Then for any $\eta > 0$ using the law of large numbers
    we can a.s.\ find a large enough $N_J$ such that
    \[
        \forall n \ge N_J,\; \Abs{\Leb(J) - f^n_J} < \eta.
    \]
    Let $N$ be large enough such that this condition is fulfilled for
    all possible combinations of intervals. 
    
    We now show that a.s.
    \[
        \forall t \ge 0, \forall n \ge N,\; d_H(I^n_t, I_t) \le \eta + \epsilon.
    \]
    Let $x \not\in I_t$, and $J_x = \OpenInterval{\ell_x, r_x}$ be the
    interval such that $x \in J$ (in case $x$ is the boundary of two
    intervals, we choose the left interval). First suppose that $\ell_x = 0$
    or $r_x = 1$. By construction $0, 1 \not\in I^n_t$, thus $d(x, 0) <
    \epsilon$ or $d(x, 1) < \epsilon$. In the other case, the variables
    $(V_i)_{i \ge 1}$ which are in $[0, \ell_x]$ and those in $[r_x, 1]$ are
    not in the same interval component of $I_t$, and by construction of
    the paintbox, their labels are not in the same block of $\Ci_t$.  For
    $n \ge 1$, let $f^n_1$ (resp.\ $f^n_2$) denote the frequency of the
    variables $(V_i)_{i \le n}$ belonging to $[0, \ell_x]$ (resp.\ $[0,
    r_x]$). The previous remark shows that there is a point $y \in
    [f^n_1, f^n_2]$ which does not belong to $I^n_t$. For $n \ge N$ we
    know that $y \in [\ell_x - \eta, r_x + \eta]$ and thus $d(x, y) \le \eta
    + \epsilon$. This shows
    \[
        \forall t \ge 0, \forall n \ge N,\; \sup_{x \not\in I_t} 
        d(x, [0, 1] \setminus I^n_t) \le \eta + \epsilon.
    \]

    Similarly consider $x_n \not\in I^n_t$. If $x_n \in \Set{0, 1}$, 
    clearly $d(x_n, [0, 1]\setminus I_t) = 0$. In the other case the
    point $x_n$ is the separation between two
    intervals of $I^n_t$. These two intervals can be seen as an
    agglomeration of blocks of size $1/n$ whose labels belong to the same
    block of $I_t$.  Let $i$ (resp.\ $j$) be the label of the right-most
    (resp.\ left-most) block of size $1/n$ of  the left interval 
    (resp.\ right interval) separated by $x_n$. The rules of the paintbox
    construction imply that $V_i$ and $V_j$ are not in the same interval
    of $I_t$, thus there exists $V_i \le y_n \le V_j$ such that $y_n
    \not\in I_t$. The value of $x_n$ is exactly the frequency of
    variables $V_1, \dots V_n$ which belong to $[0, y_n]$. Let $J_{y_n} =
    \OpenInterval{\ell_{y_n}, r_{y_n}}$ be the interval to which $y_n$
    belongs, and $f^n_1$, $f^n_2$ be as above the frequency of the $n$
    first variables in $[0, \ell_{y_n}]$ and $[0, r_{y_n}]$. As $\ell_{y_n} \le
    y_n$, we know that $f^n_1 \le x_n$, and similarly $x_n \le f^n_2$.
    Thus for $n \ge N$, $x_n \in [\ell_{y_n} - \eta, r_{y_n} + \eta]$ and
    $d(x_n, y_n) \le \eta + \epsilon$.  This shows
    \[
        \forall t \ge 0, \forall n \ge N,\; 
        \sup_{x \not\in I^n_t} d(x, [0, 1] \setminus I_t) \le \eta + \epsilon.
    \]
    Thus, a.s.\ $(I^n_t)_{t \ge 0}$ converges uniformly to $(I_t)_{t \ge 0}$.

    \medskip

    To get uniqueness, it is sufficient to notice that the distribution of the sequence
    $((I^n_t)_{t \ge 0};\, n \ge 1)$ is determined uniquely by that of
    $(\Ci_t)_{t \ge 0}$. As we can recover a.s.\ $(I_t)_{t \ge 0}$ from
    $((I^n_t)_{t \ge 0};\, n \ge 1)$,
    the distribution of $(I_t)_{t \ge 0}$ is also determined by that of $(\Ci_t)_{t \ge 0}$.
\end{proof}

\begin{remark}
    This also proves \Cref{PR:empirical} in a more detailed way. 
\end{remark}

    \subsection{Uniform nested compositions, proof of Theorem~\ref{TH:coalescentGeneral}}
    \label{SS:coalescent}

We recall that $\If$ stands for the quotient space of combs for the paintbox-equivalence
relation. To be entirely rigorous we need to define a suitable $\sigma$-field on $\If$. 
By definition of $\If$ a paintbox based on any of the representatives of a class yields
the same distribution on the space of coalescents. We can identify each
class with this distribution and endow $\If$
with the weak convergence topology of probability measures on the 
space of coalescents. We consider the associated Borel $\sigma$-field. 
This approach bears similarity with the Gromov-weak topology introduced
in \cite{greven_convergence_2006}, more on this can be found in 
\Cref{S:ultrametric}. 

The first step to find a comb representation of a given exchangeable coalescent
$(\Pi_t)_{t \ge 0}$ is to order the blocks of $(\Pi_t)_{t \ge 0}$ to
obtain a nested composition. We will do that using the notion of uniform
nested composition that we now introduce. 

\begin{definition}
    Let $(\Ci_t)_{t \ge 0}$ be an exchangeable nested composition of $\N$ and
    $(\Pi_t)_{t \ge 0}$ be the associated coalescent. We say that $(\Ci_t)_{t \ge 0}$
    is uniform if for any $n \ge 1$, conditionally on $(\Pi^n_t)_{t \ge 0}$,
    the order of the blocks of $(\Ci^n_t)_{t \ge 0}$ is uniform among all the
    possible orderings, i.e.\ all the orderings such that $(\Ci^n_t)_{t \ge 0}$
    is a nested composition.
\end{definition}

The following lemma shows that any exchangeable coalescent can be 
turned into a uniform exchangeable nested composition.

\begin{lemma} \label{lem:uniformComposition}
    Let $(\Pi_t)_{t \ge 0}$ be an exchangeable coalescent. There exists
    a uniform exchangeable nested composition $(\Ci_t)_{t \ge 0}$ whose
    associated coalescent is $(\Pi_t)_{t \ge 0}$.
\end{lemma}

\begin{proof}
    We proceed by induction. 
    For $n = 1$ there is a unique trivial possible order on
    the blocks. Suppose that we have built for $n$ an order on the
    blocks of $(\Pi^n_t)_{t \ge 0}$ such that only adjacent blocks
    can merge, we call such an order an order \emph{consistent with 
    the genealogy}. Then there are finitely many orders on the blocks
    of $(\Pi^{n+1}_t)_{t \ge 0}$ that extend the previous order
    and are consistent with the genealogy. More precisely, if 
    $n+1$ is in a block of $\Pi^{n+1}_0$ the extension is unique.
    If $n+1$ is a singleton of $\Pi^{n+1}_0$, suppose that $\Set{n+1}$
    coalesce at some point and that $k$ blocks are involved in this
    coalescence event. Then there are $k$ consistent extensions:
    $\Set{n+1}$ can be placed between any of the $k-1$ other blocks,
    or at the left-most (resp.\ right-most) position. If $\Set{n+1}$
    does not coalesce, the singleton can be placed at any position
    between blocks that do not coalesce. We pick one of these orders
    independently and uniformly.

    By induction, we have built on the same probability space as
    $(\Pi_t)_{t \ge 0}$ a nested composition of $\N$ whose blocks merge
    according to $(\Pi_t)_{t \ge 0}$.  It is easily checked from the
    construction that $(\Ci_t)_{t \ge 0}$ is a uniform nested
    composition. 
    It remains to show that it is exchangeable. Fix 
    $0 \le t_1 < \dots < t_p$, and let $c_1, \dots, c_p$ be 
    compositions of $[n]$, whose block partitions are 
    $\pi_1, \dots, \pi_n$ respectively. Fix some trajectory
    $\Pi^n \defas (\Pi^n_t)_{t \ge 0}$ of the coalescent.
    Let us denote by $O(\Pi^n)$ the number of orderings of the
    blocks of $\Pi^n_0$ yielding a nested composition, and 
    let $O(c_1, \dots, c_p; \Pi^n)$ be the number of such orderings
    verifying that $\Ci^n_{t_i} = c_i$, for $i \in \Set{1, \dots, p}$.
    Then for any permutation $\sigma$ of $[n]$, the following
    direct calculation
    \begin{align*}
        &\Prob{\Ci^n_{t_1} = c_1, \dots, \Ci^n_{t_p} = c_p}
        = \E\Big[\frac{O(c_1, \dots, c_p; \Pi^n)}{O(\Pi^n)}
        \Indic{\Pi^n_{t_1} = \pi_1, \dots, \Pi^n_{t_p} = \pi_p}\Big]\\
        &= \E\Big[\frac{O(c_1, \dots, c_p; \sigma(\Pi^n))}{O(\sigma(\Pi^n))}
        \Indic{\sigma(\Pi^n_{t_1}) = \pi_1, \dots, \sigma(\Pi^n_{t_p}) = \pi_p}\Big]\\
        &= \E\Big[\frac{O(\sigma^{-1}(c_1), \dots, \sigma^{-1}(c_p); \Pi^n)}{O(\Pi^n)}
        \Indic{\Pi^n_{t_1} = \sigma^{-1}(\pi_1), \dots, \Pi^n_{t_p} = \sigma^{-1}(\pi_p)}\Big]\\
        &= \Prob{\Ci^n_{t_1} = \sigma^{-1}(c_1), \dots, \Ci^n_{t_p} = \sigma^{-1}(c_p)}
    \end{align*}
    proves that the nested composition is exchangeable.
\end{proof}

\begin{proof}[Proof of \Cref{TH:coalescentGeneral}]
    Let $(\Pi_t)_{t \ge 0}$ be an exchangeable coalescent. Let
    $(\Ci_t)_{t \ge 0}$ be the uniform nested compositions obtained
    through \Cref{lem:uniformComposition}.  Invoking
    \Cref{TH:gnedinNested} shows that there exists a comb representation
    $(I_t)_{t \ge 0}$ of $(\Pi_t)_{t \ge 0}$. The uniqueness is immediate
    from the definition of the quotient.
\end{proof}

\section{Comb representation of \texorpdfstring{$\Lambda$}{Lambda}-coalescents}
\label{S:lambda}

In this section, we restrict our attention to the well-studied case
of $\Lambda$-coalescents. A process $(\Pi_t)_{t \ge 0}$ is a $\Lambda$-coalescent
if for any $n \ge 1$, its restriction $(\Pi^n_t)_{t \ge 0}$ to $[n]$
is a Markov process such that starting from a partition with $b$
blocks, any $k$ blocks coalesce at rate
\[
    \lambda_{b, k} = \int_{[0, 1]} x^{k-2} (1 - x)^{b-k} \Lambda(\mathrm{d}x)
\]
for a finite measure $\Lambda$ on $[0, 1]$. 

The broad aim of this section is to find a Markovian comb representation
of a given $\Lambda$-coalescent, and to provide its transitions. Recall
from the last section the path followed to obtain a comb associated to an
exchangeable coalescent. The first step is to order the block of the
coalescent to get a nested composition, and then to use
\Cref{TH:gnedinNested} to define a comb.  Here we will follow this
path in the special case of $\Lambda$-coalescents where we can have an
explicit description of both the nested composition and the comb. 

\medskip

Let us first define the nested composition associated to a $\Lambda$-coalescent.
Consider the modified transition rates
\[
    \tl_{b,k} = \frac{1}{b-k+1} \binom{b}{k} \lambda_{b,k}.
\]
Let $n \ge 1$, we define a Markov chain $(\Ci^n_t)_{t \ge 0}$ taking values
in the space of composition of $[n]$ as follows. Starting from $c$, a composition
of $[n]$ with $b$ blocks, any $k$ adjacent blocks merge at rate 
$\tl_{b,k}$. These transition rates have a natural combinatorial interpretation.
Consider $(\Pi^n_t)_{t \ge 0}$ the restriction to $[n]$ of 
a $\Lambda$-coalescent. Starting from a partition with $b$ blocks,
there are $\binom{b}{k}$ ways of merging $k$ distinct blocks. 
Thus the total transition rate from $b$ to $b-k+1$ blocks is
$\binom{b}{k} \lambda_{b,k}$.
Given that $k$ blocks merge, the blocks that merge are chosen uniformly
among the $\binom{b}{k}$ possible choices. Starting from a composition
with $b$ blocks, there are only $b-k+1$ ways to merge $k$ adjacent 
blocks. Thus, the total transition rate of $(\Ci^n_t)_{t \ge 0}$ from
$b$ to $b-k+1$ blocks is the same as $(\Pi^n_t)_{t \ge 0}$, but instead
of choosing uniformly $k$ blocks among the $\binom{b}{k}$ possibilities, we
choose $k$ \emph{adjacent} blocks among the $b-k+1$ possibilities. 

We now extend this sequence of nested compositions to a 
nested composition of $\N$. To fully determine the distribution
of $(\Ci^n_t)_{t \ge 0}$ we have to specify an initial distribution. 
We will always assume in this section that the process $(\Ci^n_t)_{t \ge 0}$
starts from the composition of $[n]$ composed of only singletons ordered
uniformly. Using the Markov projection theorem (see e.g.
\cite{kemedy_finite_1960}, Section 6.3), it is not hard to see that the 
sequence of processes $((\Ci^n_t)_{t \ge 0};\, n \ge 1)$ is sampling
consistent, i.e.\ that the restriction of $(\Ci^{n+1}_t)_{t \ge 0}$
to $[n]$ is distributed as $(\Ci^n_t)_{t \ge 0}$. Using the Kolmogorov
extension theorem we can find $(\Ci_t)_{t \ge 0}$ an exchangeable nested composition
of $\N$ whose projections to $[n]$ is distributed as $(\Ci^n_t)_{t \ge 0}$
for all $n \ge 1$. The process $(\Ci_t)_{t \ge 0}$ is a nested composition
whose blocks merge according to a $\Lambda$-coalescent. 

\begin{lemma} \label{lem:lambdaComposition}
    Let $(\Pi_t)_{t \ge 0}$ be the coalescent associated to $(\Ci_t)_{t \ge 0}$.
    Then $(\Pi_t)_{t \ge 0}$ is a $\Lambda$-coalescent. Moreover for any $t \ge 0$,
    conditionally on $\Pi^n_t$, the composition $\Ci^n_t$ is obtained by
    ordering uniformly the blocks of $\Pi^n_t$. 
\end{lemma}

\begin{proof}
    Let $(\Ci^n_t)_{t \ge 0}$ and $(\Pi^n_t)_{t \ge 0}$ be the restriction to
    $[n]$ of $(\Ci_t)_{t \ge 0}$ and $(\Pi_t)_{t \ge 0}$ respectively. Let $Q_n$
    be the generator of $(\Ci^n_t)_{t \ge 0}$ and $\hat{Q}_n$ be the generator
    of a $\Lambda$-coalescent on $[n]$. The result will follow by using
    a Markov projection theorem from~\cite{rogers1981}, see their
    Theorem~2. To apply this 
    result, we need to find a probability kernel
    $L_n$ from the space of partitions of $[n]$ to the space of compositions of $[n]$
    such that for any function $f$ from the space of partitions
    of $[n]$ to $\R$, 
    \[
        \forall \pi,\; \hat{Q}_n L_nf(\pi) = L_nQ_nf(\pi)
    \]
    and such that the initial distribution of $(\Ci^n_t)_{t \ge 0}$ is the push-forward
    by $L_n$ of the initial distribution of $(\Pi^n_t)_{t \ge 0}$. 

    Let $f$ be such a function. For $\pi$ a partition of $[n]$, let $\Ci_\pi$
    be the random composition of $[n]$ obtained by ordering the blocks
    of $\pi$ uniformly. We set
    \[
        \forall \pi,\; L_nf(\pi) = \E[f(\Ci_\pi)].
    \]
    Our choice of initial distribution for $(\Ci_t)_{t \ge 0}$ ensures that 
    the second condition holds. A straightforward generator calculation shows
    that the above equality is fulfilled and that the desired result holds.
    See \Cref{app:generator} for the details of the calculation.
\end{proof}

Using \Cref{TH:gnedinNested}, the nested composition $(\Ci_t)_{t
\ge 0}$ defines a unique nested interval-partition $(I_t)_{t \ge 0}$ that
we call the \emph{$\Lambda$-comb}. In the remainder of the section we
want to show that the $\Lambda$-comb is a Markov process and give its
transitions. We will express the transitions in terms of composition
of bridges that we now introduce.

\medskip

We say that a function $B \colon [0, 1] \to [0, 1]$ is a bridge if it is
of the form
\[
    B(x) = x(1 - \sum_{i \ge 1} \beta_i) + \sum_{i \ge 1} \beta_i \Indic{x \le V_i} 
\]
for a random mass-partition $\beta$ and an independent i.i.d.\ sequence
$(V_i)_{i \ge 1}$ of uniform $[0, 1]$-valued variables. To any
bridge we associate an interval-partition defined as 
\[
    I(B) = \interior\Big([0,1] \setminus B([0, 1]) \Big)
\]
where $B([0, 1])$ is the range of $B$. We can ask if the converse holds.
The correct notion to answer this question is that of uniform order.

\begin{definition}
    Let $I$ be a random interval-partition and $\Ci$ be the composition of $\N$
    obtained through an ordered paintbox based on $I$. We say that 
    $I$ has a uniform order if for any $n \ge 1$, the order of the 
    blocks of $\Ci \cap [n]$ is uniform. 
\end{definition}

The following lemma shows that having a uniform order is a necessary and
sufficient condition for an interval-partition to be represented by a bridge.
See \Cref{SS:proofLemmaUniform} for a proof.

\begin{lemma} \label{lem:uniformOrder}
    Let $I$ be a random interval-partition. There exists a bridge $B$
    such that $I(B) = I$ \tiff $I$ has a uniform order. If $I$ has a
    uniform order, the bridge $B$ such that $I(B) = I$ is unique in
    distribution.
\end{lemma}

Notice that for any $t \ge 0$, the $\Lambda$-comb $I_t$ at time $t$ has a
uniform order. We will denote by $B^{I_t}$ the bridge associated to $I_t$
through \Cref{lem:uniformOrder}. We are now in position to provide
the transitions of the $\Lambda$-comb.

\begin{proposition} \label{prop:lambdaGroup}
    Let $(I_t)_{t \ge 0}$ be the $\Lambda$-comb. The process $(I_t)_{t \ge 0}$
    is Markovian, and for any $s, t \ge 0$, conditionally on $I_t$,
    \begin{align} \label{eq:semigroupComb}
        I_{t+s} \overset{(\mathrm{d})}{=} I(B^{I_t} \circ B'_s)
    \end{align}
    where $B'_s$ is an independent bridge distributed as $B^{I_s}$. 
\end{proposition}

\begin{remark}
    In the coming down from infinity case we have a simpler description
    of the semi-group of the $\Lambda$-comb. Suppose that $(I_t)_{t \ge 0}$ 
    starts from an interval-partition $I_0$ with $b$ blocks and no
    dust. Then any $k$ adjacent blocks of $I_0$ merge at rate
    $\tl_{b,k}$.
\end{remark}

The above proposition shows that the $\Lambda$-comb can be represented in 
terms of composition of independent bridges. As a direct corollary, we
provide an alternative construction of the $\Lambda$-comb based on
the flow of bridges of \cite{bertoin_stochastic_2003}. 
A flow of bridges is a collection $(B_{s,t})_{s \le t}$ of bridges which
fulfills the following three conditions:
\begin{enumerate}
    \item For any $s < r < t$, $B_{s,t} = B_{s,r} \circ B_{r,t}$ (cocycle property).
    \item For any $t_1 < \dots < t_p$, the bridges $(B_{t_1, t_2}, \dots, B_{t_{p-1}, t_p})$
        are independent, and $B_{t_1, t_2}$ is distributed as $B_{0, t_2 - t_1}$ 
        (stationarity and independence of the increments).
    \item The bridge $B_{0, t}$ converges to the identity map $\Id$ as $t
        \downarrow 0$ in probability in Skorohod topology.
\end{enumerate}
It can be seen from the cocycle property that the
interval-partition-valued process $(I(B_{0, t}))_{t \ge 0}$ is a nested
interval-partition.  \cite{bertoin_stochastic_2003} have defined a
sampling procedure to obtain a coalescent from a flow of bridges. In our
context, sampling from the flow of bridges according to this procedure is
the same as doing a paintbox based on $(I(B_{0,t}))_{t \ge 0}$. An
important result from~\cite{bertoin_stochastic_2003} states that given a
$\Lambda$-coalescent $(\Pi_t)_{t \ge 0}$, there exists a unique flow of
bridges whose associated coalescent is distributed as $(\Pi_t)_{t \ge 0}$
(see Theorem~1 in~\cite{bertoin_stochastic_2003}). We call it the
$\Lambda$-flow of bridges. As a corollary of this correspondence and of
\Cref{prop:lambdaGroup}, we are able to show that the comb associated to
the $\Lambda$-flow of bridges is the $\Lambda$-comb introduced above from
the transition rates.

\begin{corollary} \label{cor:combBridge}
    Let $\Lambda$ be a finite measure on $[0, 1]$, and let $(I_t)_{t \ge 0}$
    be the $\Lambda$-comb and $(B_{s,t})_{s \le t}$ be the $\Lambda$-flow of 
    bridges. Then
    \[
        (I_t)_{t \ge 0} \overset{(\mathrm{d})}{=} (I(B_{0,t}))_{t \ge 0}.
    \]
\end{corollary}

\begin{proof}
    Let $p \ge 1$ and $0 \le t_1 < \dots < t_p$. Using the Markov
    property of $(I_t)_{t \ge 0}$ and the expression of the transitions~\eqref{eq:semigroupComb} we know that 
    \[
        (I_{t_1}, \dots, I_{t_p}) \overset{\text{(d)}}{=} 
        (I_{t_1}, I(B^{I_{t_1}} \circ B'_1), \dots, I(B^{I_{t_1}} \circ B'_1 \circ \dots \circ B'_{p-1})),
    \]
    where $(B'_1, \dots, B'_{p-1})$ are independent bridges and for $1 \le k \le p-1$, 
    $B'_k$ is distributed as $B^{I_{t_{k+1} - t_k}}$. 
    
    Let $(B_{s,t})_{s \le t}$ be the $\Lambda$-flow of bridges. Then from
    the cocycle property
    \[
        (I(B_{0,t_1}), \dots, I(B_{0,t_p})) =
        (I(B_{0, t_1}), \dots, I(B_{0, t_1} \circ B_{t_1, t_2} \circ \dots \circ B_{t_{p-1}, t_p})).
    \]
    Moreover as the flow of bridges has independent and stationary
    increments, $(B_{t_1, t_2}, \dots, B_{t_{p-1}, t_p})$ are independent
    bridges with the same distribution as above.
\end{proof}

    \subsection{Proof of Lemma~\ref{lem:uniformOrder}}
    \label{SS:proofLemmaUniform}

We will need the following continuity result. 
\begin{lemma} \label{lem:continuity}
    The map $I \colon B \mapsto I(B)$ that maps a 
    bridge to its associated interval-partition is 
    continuous when the space of interval-partitions is 
    endowed with the Hausdorff topology and the space
    of bridges with the Skorohod topology. 
\end{lemma}
\begin{proof}
    Let $B^n$ be a sequence of bridges that converge to $B$ in the
    Skorohod topology. We know that we can find a
    sequence of continuous bijections $\lambda_n$ from $\ClosedInterval{0,1}$
    to $\ClosedInterval{0,1}$ such that 
    \[
        \lim_{n \to \infty} \Norm{\lambda_n - \text{Id}}{}_{\infty} = 0
    \] 
    and 
    \[
        \lim_{n \to \infty} \Norm{B - B^n \circ \lambda_n}{}_{\infty} = 0.
    \]
    Let $I = I(B)$ and $I_n = I(B^n)$. As
    the interval-partitions are obtained from bridges, we can re-write the 
    Hausdorff distance as
    \[
        d_H(I, I_n) = 
        \sup_{x \in \ClosedInterval{0,1}} \inf_{y \in \ClosedInterval{0,1}} \Abs{B^n(x) - B(y)} 
        \vee
        \sup_{x \in \ClosedInterval{0,1}} \inf_{y \in \ClosedInterval{0,1}} \Abs{B^n(y) - B(x)}.
    \]
    We have 
    \[
        \sup_{x \in \ClosedInterval{0,1}} \inf_{y \in \ClosedInterval{0,1}}
        \Abs{B(x) - B^n(y)} \le \sup_{x \in \ClosedInterval{0,1}} \Abs{B(x) - B^n(\lambda_n(x))}
    \]
    and
    \[
        \sup_{x \in \ClosedInterval{0,1}} \inf_{y \in \ClosedInterval{0,1}}
        \Abs{B(y) - B^n(x)} \le \sup_{x \in \ClosedInterval{0,1}} \Abs{B(\lambda_n^{-1}(x)) - B^n(x)}
    \]
    and thus
    \[
        \lim_{n \to \infty} d_H(I, I^n) = 0,
    \]
    which ends the proof.
\end{proof}

\begin{proof}[Proof of \Cref{lem:uniformOrder}]
First suppose that $I$ is of the form $I(B)$ for some bridge $B$. 
Consider $B^{-1}$ the generalized inverse of $B$. Let 
$(V_i)_{i \ge 1}$ be i.i.d.\ uniform variables and $\Ci$
be the composition obtained through an ordered paintbox using
these variables. By construction of the ordered
paintbox and as $B^{-1}$ is non-decreasing, the order of the blocks of $\Ci$ is
given by the order of the variables $(B^{-1}(V_i))_{i \ge 1}$. Conditionally on
the bridge these variables are i.i.d.\ and thus their order is uniform. 

Now let $I$ be an interval-partition with a uniform order and $\Ci$ be
the composition obtained by an ordered paintbox. We will first consider
the case where $I$ has finitely many interval components and no dust. The
fact that the order of the blocks of the composition $\Ci$ is uniform
shows that the order of the interval components of $I$ is uniform (each
block of $\Ci$ corresponds to an interval of $I$). Let $K$ be the number
of blocks of $I$, and let $V^*_1 < \dots < V^*_K$ be the order statistics
of independent uniform variables. Suppose that $\beta_1$ is the length of
the left-most interval of $I$, $\beta_2$ that of the second left-most,
etc.\ then
\[
    \forall u \in [0, 1],\; B(u) = \sum_{i = 1}^K \beta_i \Indic{V^*_i \le u}
\]
is a bridge such that $I(B) = I$. Indeed, since the order of the
intervals is uniform, there is a uniform permutation $\sigma$ of $[K]$
independent of $V^*_1,\dots, V^*_K$, such that $(\beta_{\sigma(i)})$ is
ranked in nonincreasing order. This shows that 
\[
    B(u) = \sum_{i = 1}^K \beta_{\sigma(i)} \Indic{V^*_{\sigma(i)} \le u}
\]
indeed defines a bridge. This also shows the uniqueness in distribution
of $B$.

Let us turn to the general case. Let $n \ge 1$ and consider $I^n$ the
empirical interval-partition associated to $\Ci \cap [n]$. By assumption
the interval-partition $I^n$ has a uniform order, thus using the above
argument we can find a unique bridge $B^n$ such that $I(B^n) = I^n$. We
know that $I^n$ converges a.s.\ to $I$. Let $\beta_n$ (resp.\ $\beta$) be
the mass-partition associated to $I^n$ (resp.\ $I$). As the function that
maps an interval-partition to its mass-partition is continuous, we have
that $\beta_n$ converges a.s.\ to $\beta$ (see e.g.~\cite{bertoin_2006}
Proposition~2.2).
We can now make use of another continuity result, namely Lemma~1
from~\cite{bertoin_stochastic_2003}, to show that the sequence of bridges
$(B^n)_{n \ge 1}$ converges in distribution to a bridge $B$ obtained from
the mass-partition $\beta$. Using \Cref{lem:continuity}, we know
that $I(B^n)$ converges in distribution to $I(B)$. By uniqueness of the
limit, we get that
\[
    I \overset{(\text{d})}{=} I(B),
\]
and that $B$ is unique.
\end{proof}

    \subsection{Proof of Proposition~\ref{prop:lambdaGroup}}
    \label{SS:proofSemigroup}

We will first prove \Cref{prop:lambdaGroup} for empirical 
interval-partitions and then take the limit. We start by proving the 
following lemma, which is the direct reformulation of
\Cref{prop:lambdaGroup} for empirical interval-partitions.

\begin{lemma} \label{lem:empiricalLambdaTransition}
    Let $\Ci^n_0$ be an exchangeable composition of $[n]$ with a uniform
    order on its blocks, and let $(\Ci^n_t)_{t \ge 0}$ be the Markov
    process started from $\Ci^n_0$ with transitions $(\tl_{b,k};\, 2 \le
    k \le b < \infty)$. If $(I^n_t)_{t \ge 0}$ denotes the empirical
    nested interval-partition associated to $(\Ci^n_t)_{t \ge 0}$, then
    conditionally on $\Ci^n_0$,
    \[
        I^n_t \overset{\mathrm{(d)}}{=} I(B_0^n \circ B_t),
    \]
    where $B_0^n$ and $B_t$ are independent bridges such that 
    $I(B^n_0) = I^n_0$ and $I(B_t) = I_t$, the $\Lambda$-comb at time
    $t$.
\end{lemma}

\begin{proof}
    Let us denote by $(A_1, \dots, A_K)$ the blocks of $\Ci^n_0$ in
    order of their least element. As $\Ci^n_0$ has a uniform order on
    its blocks, according to \Cref{lem:uniformOrder} we can find 
    $(U_1, \dots, U_K)$ such that conditionally on $K$ these are 
    i.i.d.\ uniform variables on $[0, 1]$ and 
    \[
        \forall r \in [0, 1],\; B^n_0(r) = \frac{1}{n} \sum_{i = 1}^K
        \Card(A_i) \Indic{U_i \le r}
    \]
    defines a bridges satisfying $I(B^n_0) = I^n_0$. Let $B_t$ be
    independent and such that $I(B_t) = I_t$. To each interval component
    of $I^n_0$ corresponds a unique block $A_i$ of $\Ci^n_0$, and thus a
    unique jump time $U_i$ of $B^n_0$. We claim that $I(B^n_0 \circ B_t)$
    is obtained by merging the intervals of $I^n_0$ whose jump times
    belong to the same interval component of $I_t$. To see this, notice
    that by definition $I(B^n_0 \circ B_t)$ is the set of flats of 
    $(B^n_0 \circ B_t)^{-1} = B^{-1}_t \circ (B^n_0)^{-1}$. Thus $x$ and 
    $y$ belong to the same flat of $(B^n_0 \circ B_t)^{-1}$ \tiff 
    $(B^n_0)^{-1}(x)$ and $(B^n_0)^{-1}(y)$ belong to the same flat of 
    $B_t^{-1}$, that is to the same interval component of $I_t$. The
    claim is proved by further noting that $(B^n_0)^{-1}(x)$ is the jump time
    of the interval component of $I^n_0$ to which $x$ belongs.

    The previous procedure can be rephrased in terms of an ordered
    paintbox. The interval-partition $I(B^n_0 \circ B_t)$ is obtained by
    labeling uniformly the $K$ blocks of $I^n_0$, sampling a composition
    $\Ci'_t$ of $[K]$ according to an ordered paintbox based on $I_t$ and
    merging the intervals of $I^n_0$ whose labels belong to the same
    block of $\Ci'_t$. As $I_t$ is the $\Lambda$-comb at time $t$, the 
    composition $\Ci'_t$ is distributed as $\Ci^K_t$, the nested
    composition at time $t$ obtained by merging $K$ initial singleton blocks
    ordered uniformly according to the rates $(\tl_{b,k};\, 2 \le k \le b
    < \infty)$. Thus $I(B^n_0 \circ B_t)$ can be obtained by letting its
    intervals merge at rate $(\tl_{b,k};\, 2 \le k \le b < \infty)$,
    and is distributed as $I^n_t$.
\end{proof}

\begin{proof}[Proof of \Cref{prop:lambdaGroup}]
    Let $(I_t)_{t \ge 0}$ be the $\Lambda$-comb, and 
    $(V_i)_{i \ge 1}$ be an independent sequence of i.i.d.\ uniform
    variables on $[0, 1]$. Denote by $(\Ci^n_t)_{t \ge 0}$ the nested
    composition of $[n]$ obtained by an ordered paintbox based on
    $(I_t)_{t \ge 0}$ using the sampling variables $(V_i)_{i \ge 1}$, and 
    let $(I^n_t)_{t \ge 0}$ be the corresponding empirical nested
    interval-partition. According to \Cref{lem:lambdaComposition} the
    interval-partition $I_t$ has a uniform order, and thus there exists a
    bridge $B_t$ such that $I(B_t) = I_t$. Conditionally on $B_t$, 
    the sequence
    \[
        \forall i \ge 1,\; \xi_i = B_t^{-1}(V_i)  
    \]
    is i.i.d. We denote by $\mu_t$ the (random) law of $\xi_1$
    conditionally on $B_t$, and by $\mu^n_t$ its empirical distribution
    defined as
    \[
        \mu^n_t = \frac{1}{n} \sum_{i = 1}^n \delta_{\xi_i}.
    \]
    Note that $B_t$ is the distribution function of $\mu_t$. If $B^n_t$
    denotes the distribution function of $\mu^n_t$, then $B^n_t$ is a
    bridge such that $I(B^n_t) = I^n_t$. It follows from
    \Cref{lem:empiricalLambdaTransition} that
    \begin{align} \label{eq:empiricalGroup}
        (I^n_t, I^n_{t+s}) \overset{\mathrm{(d)}}{=} (I^n_t, I(B^n_t \circ B'_s))
    \end{align}
    where $B'_s$ is an independent bridge distributed as $B^{I_s}$.
    The result will follow by taking the limit
    in~\eqref{eq:empiricalGroup}.

    According to the Glivenko-Cantelli theorem (see Proposition~4.24 
    in \cite{kallenberg_foundations_2002}), the sequence of bridges 
    $(B^n_t)_{n \ge 1}$ converges almost surely to $B_t$ in the uniform
    topology. Thus $B^n_t \circ B'_s$ converges a.s.\ in the uniform
    topology to $B_t \circ B'_s$, and by \Cref{lem:continuity} and 
    \Cref{PR:empirical} the right-hand side of~\eqref{eq:empiricalGroup}
    converges a.s.\ to $(I_t, I(B_t \circ B'_s))$. According to
    \Cref{PR:empirical}, the left-hand side converges a.s.\ to 
    $(I_t, I_{t+s})$ and we have proved that~\eqref{eq:semigroupComb} holds.
        
    It remains to show that $(I_t)_{t \ge 0}$ is Markovian. It is
    sufficient to prove that $(\Ci_t)_{t \ge 0}$ is Markovian.
    This follows from standard arguments from measure theory by noting that
    $\sigma$-field of $(\Ci_t)_{t \ge 0}$ is induced by that of its
    restrictions to $[n]$, and that all of these restrictions are Markov.
\end{proof}

    \subsection{Dynamical combs}
    \label{SS:dynamicalComb}

As mentioned in the introduction, an exchangeable coalescent models the
genealogy of a population observed at a given time. By varying the
observation time we obtain a dynamical genealogy that has been named the
\emph{evolving coalescent}.  There has been much interest into studying
evolving coalescents.  For example, if the coalescent at a fixed time is
the Kingman coalescent, the authors of~\cite{pfaffelhuber_process_2006,
pfaffelhuber_tree_2011} have studied statistics of the evolving
coalescent using a look-down representation, the authors
of~\cite{greven_tree-valued_2008} studied the dynamics of the entire tree
structure using the framework of the Gromov-weak topology. Evolving
coalescents such that the coalescent at a fixed time is a more general
$\Lambda$-coalescent have also been considered, see e.g.\
\cite{kersting2014} for the case of Beta-coalescents
and~\cite{schweinsberg2012} for the Bolthausen-Sznitman coalescent. 

In this section we show that the previous results on the Markov property
of the $\Lambda$-comb allow us to define a comb-valued process, the
\emph{evolving comb}, such that sampling from the evolving comb at a
fixed time yields a $\Lambda$-coalescent. The evolving comb contains all
the information about the dynamical genealogy but does not require the
cumbersome framework of random metric spaces endowed with the
Gromov-Hausdorff topology as in~\cite{greven_tree-valued_2008}. For the
sake of clarity we will only consider the evolving Kingman comb where we
have an explicit construction of the genealogy at a fixed time.

\medskip

We will build the evolving Kingman comb by defining its semi-group.
Recall that when the coalescent associated to a nested interval-partition
comes down from infinity, the comb can be represented using a comb
function, see \Cref{SS:introComb}.  Let $f$ be a deterministic comb
function and $s > 0$, we want to describe the genealogy of the population
at time $s$ given that its genealogy at time $0$ is encoded by $f$. The
procedure we follow is illustrated in \Cref{F:composition}.  Recall the
Kingman comb construction discussed in introduction. Let $(e_i)_{i \ge
1}$ be a sequence of i.i.d.\ exponential variables, and $(U_i)_{i \ge 1}$
a sequence of i.i.d.\ uniform $[0, 1]$ variables. For $i \ge 1$, we set
\[
    T_i = \sum_{k \ge i+1} \frac{2}{k(k-1)} e_k.
\]
The Kingman comb is given by 
\[
    f_K = \sum_{i \ge 1} T_i \Indic*{U_i}.
\]
It is known from~\cite{lambert_recovering_2016}, Proposition~3.1,  that
the above construction generates the comb associated to the flow of
bridges, i.e.\ the $\Lambda$-comb associated to the Kingman coalescent.
There are only finitely many teeth of $f_K$ that are larger than $s$,
i.e.\ such that $T_i \ge s$, say $N_s$. Let $\sigma$ be their order,
e.g.\ $\sigma(1)$ is the label of the left-most tooth. Consider $V^*_1 <
\dots < V^*_{N_s + 1}$ the order statistics of $N_s +1$ independent
i.i.d.\ uniform variables. For $1 \le k \le N_s$ let $M_k$ be the
greatest tooth of $f$ in the interval $(V^*_k, V^*_{k+1})$, i.e.
\[
    M_k = \sup_{(V^*_k, V^*_{k+1})} f. 
\]
We define new variables $(\hat{T}_i)_{i \ge 1}$ as follows
\[
    \forall i > N_s,\; \hat{T}_i = T_i,
\]
and 
\[
    \forall i \le N_s,\; \hat{T}_{\sigma(i)} = M_i + s.
\]
We define
\[
    \hat{f}_K = \sum_{i \ge 1} \hat{T}_i \Indic*{U_i}.
\]
Geometrically, the comb $\hat{f}_K$ is obtained through a cutting and
pasting procedure illustrated in \Cref{F:composition}. 

The above construction defines an operator given by
\[
    P_t F(f) = \E[F(\hat{f}_K)],
\]
for all continuous bounded functions $F$. We will show below that the
family of operators $(P_t)_{t \ge 0}$ is a semi-group. Thus we can define
a comb-valued Markov process $(\Ii^r)_{r \ge 0}$ whose transitions are
given by the above construction. We call the process $(\Ii^r)_{r \ge 0}$
the \emph{evolving Kingman comb}. 
\begin{lemma}
    The family of operators $(P_t)_{t \ge 0}$ is a semi-group.
    Moreover the Kingman comb is a stationary distribution of the
    evolving Kingman comb.
\end{lemma}
\begin{proof}
    Let $s,t \ge 0$, let $f$ be a deterministic comb. We call $f_t$ the
    comb obtained through the above procedure at level $t$ starting from $f$, 
    and $f_{t+s}$ the one obtained according to the above procedure at level
    $s$, but using $f_t$ as starting comb. We need to show that $f_{t+s}$ is
    distributed as $f'_{t+s}$, the comb obtained at level $t+s$ starting from $f$. 

    It is sufficient to show that the portion of the comb $f_{t+s}$ lying
    between level $0$ and $t+s$ is distributed as a Kingman comb truncated at
    height $t+s$. To show that, it is more convenient to see combs as nested
    interval-partitions. The procedure described above can be rephrased in terms 
    of composition. Suppose that $f_{t+s}$ has $K$ truncated teeth at time
    $s$, this defines $K+1$ intervals of $[0, 1]$. For each of these intervals of 
    $f_{t+s}$, we throw a uniform variable. Two intervals merge at the first moment
    when their corresponding variables belong to the same subinterval of $f_t$.
    This is exactly the description of the ordered paintbox procedure. Thus,
    using the Markov property of the Kingman comb we know that $f_{t+s}$,
    between level $0$ and $t+s$, is distributed as the truncation of a Kingman
    comb. This argument also shows that the Kingman comb is a stationary distribution.
\end{proof}

\begin{figure}[t]
    \begin{center}
    \includegraphics[width=\textwidth]{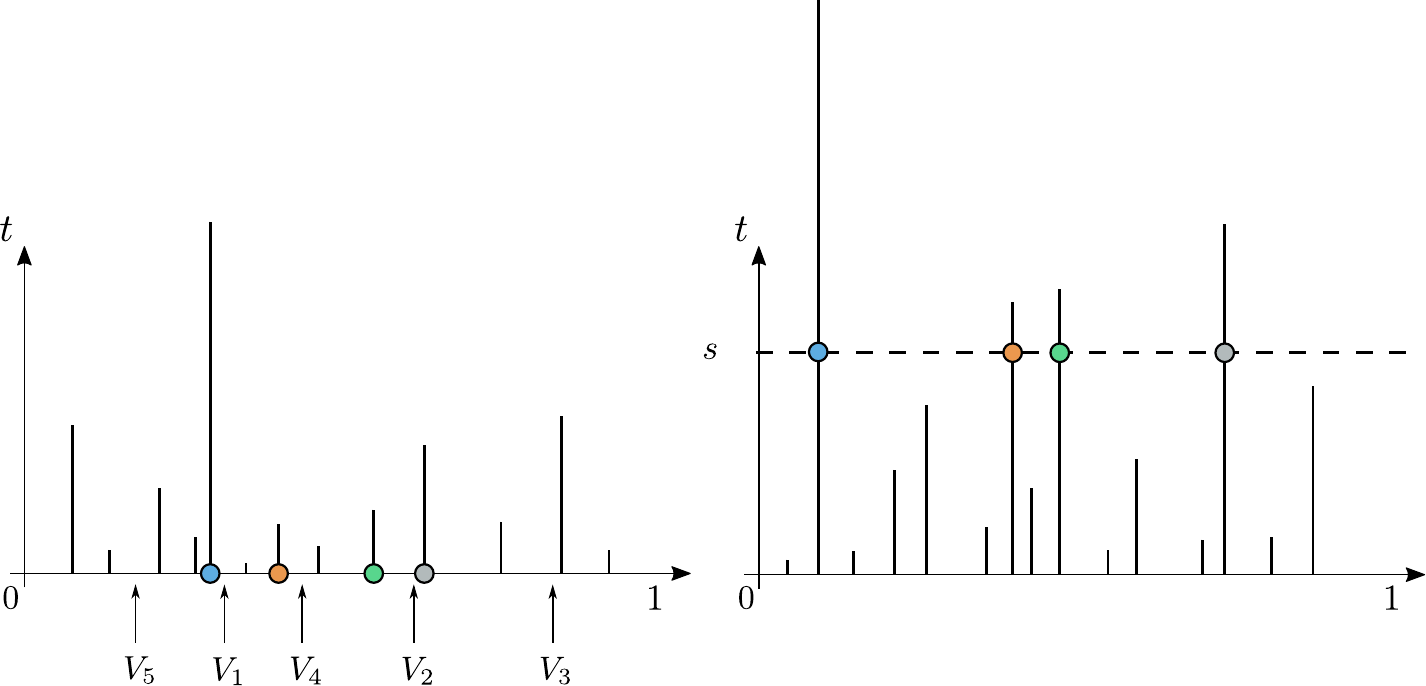}
    \caption{Transition of the evolving Kingman comb. The
    comb at time $s$, $\hat{f}_K$, is represented on the right, and the initial comb
    $f$ is on the left. To obtain $\hat{f}_K$, one has first to erase
    the part of the right comb lying above level $s$. Here we have
    erased $N_s = 4$ teeth. Then throw $N_s + 1$ uniform variables
    $V_1, \dots, V_{N_s + 1}$, this defines $N_s$ intervals between
    these variables, here $(V_5, V_1), (V_1, V_4)$, $(V_4, V_2)$ and $(V_2, V_3)$. 
    Finally take the largest tooth of $f$ in each of these
    intervals, represented with a coloured root, and paste it in place of
    the erased tooth.}
    \label{F:composition}
    \end{center}
\end{figure}

\medskip

This construction can be easily extended to the case of
$\Lambda$-coalescents that come down from infinity, even though we do not
have an explicit construction of the comb in this case. In short, to
obtain the evolving comb at time $s$, one needs to sample independently a
new comb, erase the portion lying above height $s$ and replace it by
teeth sampled from the original comb. In the general case, we have to
define the transition of the evolving comb using composition of bridges.

Again, the evolving comb can be built from the flow of bridges. Let
$(B_{s,t})_{t \ge 0}$ be a $\Lambda$-flow of bridges, for any time $r$ we
can build a nested interval-partition by setting 
\[
    (I^r_t)_{t \ge 0} = (I(B_{r, r+t}))_{t \ge 0}. 
\]
Then, using a similar argument as in the proof of \Cref{cor:combBridge}
we could show that the comb-valued process $(\Ii^r)_{r \ge 0} = (
(I^{-r}_t)_{t \ge 0} )_{r \ge 0}$ is distributed as the evolving comb
introduced above. As a remark this provides a c\`adl\`ag modification of
the evolving comb, and the Feller property of the flow of bridges ensures
that the evolving comb is a Feller process.

\section{Combs and ultrametric spaces}
\label{S:ultrametric}

In this section we envision combs as random UMS. Random metric
measure spaces have already been studied
in~\cite{greven_convergence_2006, gromov2007metric}.  A key working
hypothesis there is that the metric spaces are \emph{separable}.  In
terms of combs and coalescents, separability translates into absence of
dust (see \Cref{SS:separable}). While separability is a very
natural hypothesis when considering metric measure spaces,
restricting our attention to combs without dust seems arbitrary, as dust
has not raised any difficulty so far. In this section we provide a
straightforward extension of the framework of random metric measure
spaces to account for non-separable UMS.

Let us recall the heuristic of our approach and give a short outline of
this section. After a discussion on the assumptions of \Cref{def:UMS} in
\Cref{SS:generalUMS}, we define a topology on the space of UMS in
\Cref{SS:gromovWeak} by saying that a sequence of UMS converges if the
associated sequence of coalescents converges weakly as probability
measures.  In the separable case, the Gromov reconstruction theorem
(see Section~3.$\frac{1}{2}$.5 of~\cite{gromov2007metric}) ensures
that spaces that are indistinguishable have the support of their
measures in isometry. In general this result does not hold, we want to
obtain a similar result for general UMS. In order to do that, we
introduce in \Cref{SS:backbone} the notion of a backbone of a UMS. An
UMS can be seen as the leaves of a tree. This tree can be decomposed into
1) a separable part, that we call the backbone and 2) additional subtrees
grafted on this backbone. Even though these subtrees can have a complex
geometry, from a sampling standpoint they behave as star-trees (recall
\Cref{F:quotient}). In \Cref{SS:samplingEquivalence}, we show that if two
UMS are indistinguishable in the Gromov-weak topology, then they are
weakly isometric, in the sense that we can find an isometry between their
backbones and a measure-preserving correspondence between the star-trees
attached to them (see \Cref{prop:samplingEquivalence} for a rigorous
statement).  Finally \Cref{SS:combCompletion} is dedicated to showing
\Cref{cor:representationGeneral}, i.e.\ that we can always find a comb
metric space weakly isometric to a given UMS with complete
backbone, and \Cref{SS:separable} is devoted to showing
\Cref{cor:isometrySeparableUMS} and
\Cref{prop:representationSeparable} which are the analogous
results in the complete and separable case.

    \subsection{Discussion of Definition~\ref{def:UMS}}
    \label{SS:generalUMS}

Recall \Cref{def:UMS} of a UMS from the introduction. This
definition has two differences with the ``naive'' definition of a UMS
(that is, any ultrametric space endowed with a probability measure on its Borel
$\sigma$-field). First, we impose a measurability condition on the metric
$d$. Second we allow the measure $\mu$ to be defined on a $\sigma$-field
that is smaller than the usual Borel $\sigma$-field. In this section, we
start with a discussion of the assumptions of \Cref{def:UMS}.

\medskip

Let $\Pcoal$ denote the state space of coalescents, endowed with its
usual Borel $\sigma$-field (see~\cite{bertoin_2006}
Lemma~2.6), and let $\Pi$ be the application defined as
\[
    \Pi \colon
    \begin{dcases}
        U^{\N} \to \Pcoal\\
        (x_i)_{i \ge 1} \mapsto (\Pi_t)_{t \ge 0},
    \end{dcases}
\]
where
\[
    i \sim_{\Pi_t} j \iff d(x_i, x_j) \le t.
\]
The following simple lemma proves that the measurability of $d$ is the
minimal requirement so that the coalescent obtained by sampling from $U$
is a measurable process.

\begin{lemma}
    The application $\Pi$ is measurable when $U^{\N}$ is endowed with the
    product $\sigma$-algebra $\Ui^{\otimes \N}$ \tiff the distance $d$
is $\Ui \otimes \Ui$ measurable.
\end{lemma}

\begin{proof}
    Notice that by definition of $\Pi$ we have
    \[
        \Set{d(x_1, x_2) \le t} = \Set{1 \sim_{\Pi_t} 2}
    \]
    which yields the ``only if'' part of the proof.
    
    \medskip

    To prove the converse implication, let $\pi$ be a partition of $[n]$
    and define
    \[
        R_{i,j} = 
        \begin{dcases}
            \Set{d(x_i, x_j) \le t} &\text{ if $i \sim_\pi j$,}\\
            \Set{d(x_i, x_j) > t} &\text{ if $i \not\sim_\pi j$.}
        \end{dcases}
    \]
    Then
    \[
        \Set{\Restriction{\Pi_t}{[n]} = \pi} = \bigcap_{i, j \le n} R_{i,j},
    \]
    which ends the proof.
\end{proof}

We now turn to the second point of the definition. Roughly speaking,
the Borel $\sigma$-field of a non-separable ultrametric space tends to be
large, and fewer measures can be defined on it. It is natural to ask
whether all coalescents (especially coalescents with dust) can be
represented as samples from ultrametric measure spaces, endowed with
their natural Borel $\sigma$-field. It turns out that this
question can be linked to a deep measure-theoretic problem known as the
Banach-Ulam problem. It can be formulated as follows: can we find a space
$X$ and a probability measure $\mu$ defined on the power set of $X$ such
that $\mu(\Set{x}) = 0$ for all $x \in X$? Point~(iii) of the following
proposition yields a positive answer to this question.

\begin{proposition} \label{prop:banachUlam}
    The following statements are equivalent.
    \begin{enumerate}
        \item[\rm(i)] There exists an exchangeable coalescent with dust that can be obtained as
            a sample from a Borel UMS.
        \item[\rm(ii)] Any exchangeable coalescent can be obtained as a sample from a Borel
            UMS.
        \item[\rm(iii)] There exists an extension of the Lebesgue measure to all
            subsets of $\R$.
    \end{enumerate}
\end{proposition}

This proposition is proved in \Cref{app:banachUlam}. A treatment of
the Banach-Ulam problem requires advanced tools from set theory. Let us 
briefly report some basic facts about it, and refer the interested reader 
to \cite{fremlin_real_1993}.

The last point of \Cref{prop:banachUlam} cannot be proved from the
sole axioms Zermelo-Fraenckel-Choice (ZFC) of set theory. This is a
consequence of the following two well-known facts. First, assuming
the continuum hypothesis (CH) and the axiom of choice it is possible
to prove that there exists no extension of the Lebesgue measure to
all subsets of $\R$, see e.g.\ the end of Section~3 of Chapter~1
of~\cite{billingsley_1995}. Second, ZFC and ZFC+CH are
equiconsistent. Thus a proof in ZFC that an extension of the Lebesgue
measure exists would show that ZFC+CH is inconsistent, and thus that ZFC
itself is inconsistent.

The previous paragraph shows that a positive answer to our question
cannot be obtained from ZFC alone. We can now ask if a negative
answer could be proved. However, as we will see it cannot be shown
that a negative answer to our question is not provable in ZFC.
According to Corollary~2E
in~\cite{fremlin_real_1993}, the theory ZFC+``there exists an extension
of the Lebesgue measure'' is equiconsistent with the theory ZFC+``there
exists a measurable cardinal''. Measurable cardinals are instances of
(strongly) inaccessible cardinals (see Theorem~1D
of~\cite{fremlin_real_1993}), and it is well-known that
it cannot be shown that the existence of inaccessible cardinals is
consistent relative to ZFC, see e.g.\ Theorem~12.12 in~\cite{jech_2003}.

Summing up these two consistency results, the situation is the
following: assuming that ZFC is consistent, we can safely suppose
that no coalescent with dust can be obtained as a sample from a Borel
UMS. However it cannot be shown from ZFC that assuming the converse
statement will not lead to a contradiction. According to
\cite{fremlin_real_1993} such a scenario is extremely unlikely, as many
consequences of ZFC+``there exists a measurable cardinal'' have been
explored and have not led to any contradiction so far, see the discussion
in Remark~1E(e) in~\cite{fremlin_real_1993}.

Obviously, all these considerations go far beyond the scope of the
current work. The approach we propose in this paper is to let the 
sampling measures be defined on a smaller $\sigma$-field, namely $\Ui$.
This relaxation allows us to find enough measures to represent all
coalescents as samples from ultrametric measure spaces, avoiding
the aforementioned set-theoretic issues. We simply hope that this short
discussion has led the reader to the conclusion that allowing sampling
measures to be defined on $\Ui$ is a more natural framework in which 
discussing coalescent theory on non-separable UMS than having to assume
that one of the statements of \Cref{prop:banachUlam} holds.

\medskip

Let us finally discuss the last point of \Cref{def:UMS}. This point
can be reformulated in terms of the ball $\sigma$-field which is defined
as follows.

\begin{definition} \label{def:ballField}
    Let $(U, d)$ be an ultrametric space. The ball $\sigma$-field
    denoted by $\ball$ is the $\sigma$-field induced by the open balls of 
    $(U, d)$, that is,
    \[
        \ball = \sigma( \Set{B(x, t) \suchthat x \in U, t > 0} ),
    \]
    where
    \[
        \forall x \in U, \forall t > 0,\; B(x, t) = \Set{y \in U \suchthat d(x,y) < t}.
    \]
\end{definition}

\begin{example} \label{ex:countable}
    Consider any set $U$ endowed with the metric
    \[
        \forall x,y \in U,\; d(x,y) = \Indic{x \ne y}.
    \]
    In this case $\ball$ is the countable-cocountable $\sigma$-field.
\end{example}

The last point of \Cref{def:UMS} can now be rephrased as
$\ball \subseteq \Ui \subseteq \Bi(U)$. It is important to notice that if
$\Bi(U)$ denotes the Borel $\sigma$-field of $(U, d)$, then 
$\ball \subseteq \Bi(U)$ always holds. In that sense, our definition
of a UMS should be seen as a generalization of the naive definition
as more measures can be defined on $\ball$ than on $\Bi(U)$. The
converse statement, i.e.\ that $\Bi(U) \subseteq \ball$, does
not hold in general, as \Cref{ex:countable} shows. Nevertheless,
in the important case where $(U, d)$ is separable, we have that
$\ball = \Bi(U)$, and the ultrametric $d$ is $\Bi(U) \otimes
\Bi(U)$-measurable. We thus recover the usual framework of 
metric measure spaces.

\begin{remark} \label{rem:ballField}
    The ball $\sigma$-field appears in other contexts
    where the underlying metric space is not separable, for example when
    considering the space of c\`adl\`ag functions with the uniform
    topology, as in~\cite{billingsley_1999}, Section~6 and Section~15.    
\end{remark}

    \subsection{The Gromov-weak topology}\label{SS:gromovWeak}

We now define the Gromov-weak topology on the space of UMS. Let $(U, d,
\Ui, \mu)$ be a UMS, and consider $(X_i)_{i \ge 1}$ an i.i.d.\ sequence
distributed as $\mu$. Recall that we define an exchangeable coalescent through
the set of relations
\[
    i \sim_{\Pi_t} j \iff d(X_i, X_j) \le t. 
\]
Alternatively, we can see this coalescent as a random pseudo-ultrametric
on $\N$ defined as
\[
    \forall i,j \ge 1,\; d_\Pi(i, j) = d(X_i, X_j). 
\]
Both objects encode the same information, as $d_\Pi$ can be recovered
from $(\Pi_t)_{t \ge 0}$ through the equality
\[
    \forall i,j \ge 1,\; d_\Pi(i, j) = \inf \Set{t \ge 0 \suchthat i \sim_{\Pi_t} j}. 
\]
The distribution of this pseudo-ultrametric is called the \emph{distance
matrix distribution} of the UMS. 

\begin{remark}
    From a topological point of view, a pseudo-ultrametric on $\N$ can be seen
    as an element of $\R_+^{\N\times\N}$ endowed with its
    product topology. Notice that in this case, the correspondence between
    pseudo-ultrametrics on $\N$ and coalescents outlined above is a homeomorphism.
\end{remark}

We use distance matrix distributions to define a topology on the space of
UMS. Consider a sequence $(U_n, d_n, \Ui_n, \mu_n)_{n \ge 1}$ of UMS, and
denote by $(\nu_n)_{n \ge 1}$ the associated sequence of distance matrix
distributions. We say that the sequence $(U_n, d_n, \Ui, \mu_n)_{n \ge
1}$ converges in the Gromov-weak topology to $(U, d, \Ui, \mu)$ if
$(\nu_n)_{n \ge 1}$ converges weakly to $\nu$, the distance matrix
distribution of $(U, d, \mu)$, in the space of probability measures on
$\R_+^{\N\times\N}$.

    \subsection{Backbone} \label{SS:backbone}

It is well known that any ultrametric space $(U, d)$ can be seen as the
leaves of a tree. This is illustrated in \Cref{F:quotient}.
Formally, we work on the space $U \times \R_+$ and consider the
pseudo-metric
\[
    d_T\big((x,s), (y,t)\big) = \max\Big(d(x,y) - \frac{s+t}{2}, \frac{\Abs{t-s}}{2}\Big).
\]
Let $T$ be the space $U \times \R_+$ quotiented by the equivalence relation
\[
    z \sim z' \iff d_T(z,z') = 0.
\] 
Then the space $(T, d_T)$ is a real tree
(see~\cite{evans_probability_2007}, Definition~3.15) whose
leaves can be identified with $(U, d)$.  

\begin{definition}[Backbone of $T$] \label{def:backbone}
    Define 
    \[
        f \colon \begin{cases}
            U \to \R_+\\
            x \mapsto \inf \Set{t \ge 0 \suchthat \mu(B(x, t)) > 0},
        \end{cases}
    \]
    (note that $f$ is measurable since $\ball \subseteq \Ui$) and
    let 
    \[
        \cS \defas \Set{(x,t) \in T \suchthat t \geq f(x)}.
    \] 
    The space $\cS$ will be
    referred to as the backbone of the tree $T$, and we denote by $d_{\cS}$ the
    distance $d_T$ restricted to $\cS$.
\end{definition}

Let us now motivate the next result that will be fundamental to our
approach.  In words, \Cref{prop:decomposition} states that even if the
underlying UMS is not separable, the backbone is always a separable tree.
Secondly, one can recover the whole tree from the backbone by grafting
some ``simple'' subtrees on the skeleton. By ``simple'', we mean that
each of those subtrees has the sampling properties of a star-tree. Let us
be more explicit about this last statement and discuss an example. 

Consider the space $[0, 1] \times \Set{0, 1}$ endowed with the
ultrametric
\begin{align*}
    \forall x,y \in [0, 1],\; \forall a,b \in \Set{0, 1},\; d\big( (x,a), (y,b) \big) = 
    \begin{cases}
        1 &\text{ if $x \ne y$},\\
        1/2 &\text{ if $x = y$ and $a \ne b$},\\
        0 &\text{ if $(x,a) = (y,b)$}.
    \end{cases}
\end{align*}
The space $([0, 1] \times \Set{0, 1}, d)$ is a star-tree where
each branch splits in two at height $1/2$ (see \Cref{F:bifurcating}
left panel), we call it the bifurcating star-tree. We 
endow this space with the product measure of the Lebesgue measure
on $[0, 1]$ and the uniform measure on $\Set{0, 1}$, defined on the 
usual product Borel $\sigma$-field. Consider
two independent random variables $(X, A)$ and $(Y, B)$ distributed
according to the above measure. We see that these two variables 
lie at distance $1/2$ \tiff $X = Y$ and $A \ne B$, which happens
with probability $0$. Thus, from a sampling point of view,
all points of the space lie at distance $1$ from one another,
i.e.\ the bifurcating star-tree is a star-tree (see \Cref{F:bifurcating}
right panel). 

\begin{figure}
    \center
    \includegraphics{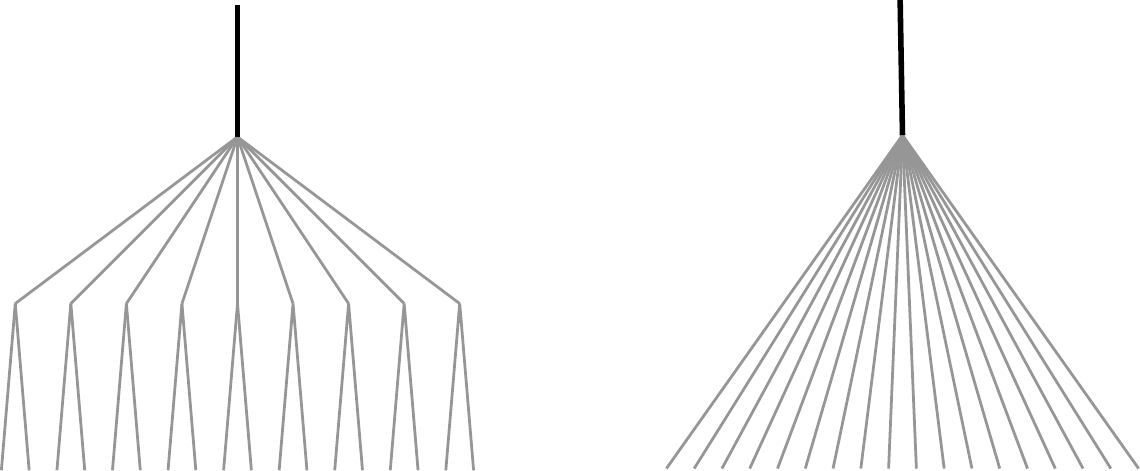}
    \caption{Left panel: The bifurcating star-tree. 
    Right panel: The bifurcating star-tree simplified
    according to the metric $\tilde{d}$. In both cases,
    the backbone is illustrated with a bold black line
    and the subtrees attached to it with thin grey lines.}
    \label{F:bifurcating}
\end{figure}

This examples illustrates the more general phenomenon that from the
measure point of view, the subtrees attached to the backbone 
behave like star-trees. More formally, consider an
UMS $(U, d, \Ui, \mu)$. We introduce the distance
\[
    \forall x,y \in U,\; \tilde{d}(x,y) = \Indic{x \ne y} \inf \Set{t \ge 0 \suchthat 
        d(x,y) \le t \text{ and } \mu(B(x,t)) > 0},
\]
which replaces each subtree attached to the backbone by 
a star-tree. The point (iii) of the following proposition
shows that the coalescent obtained by sampling from 
$(U, d, \Ui, \mu)$ is the same as the coalescent obtained
by sampling from $(U, \tilde{d}, \Ui, \mu)$. 

\begin{proposition}
    \label{prop:decomposition}
    \begin{enumerate}
        \item[\rm (i)] The space $(\cS, d_{\cS})$ is a separable real tree.

        \item[\rm(ii)] The map
        \begin{eqnarray*}
            \psi \colon 
            \begin{cases}
                (U, \Ui) \to (\cS, \Bi(\cS)) \\
                x \mapsto (x, f(x))  
            \end{cases}
        \end{eqnarray*}    
        is measurable and we define $\mu_{\cS} \defas \psi \star \mu$, the
        pushforward measure (on $(\cS, \Bi(\cS))$) of $\mu$ by $\psi$. In
        particular, the support of $\mu_{\cS}$ belongs to the subset of the
        backbone $\Set{(x,t) \in \cS \suchthat t=f(x)}$.

        \item[\rm(iii)] Consider an i.i.d.\ sequence $(X_i)_{i \ge 1}$ 
        distributed according to $\mu$. Then a.s.\ for all $i, j \ge 1$,
        $\tilde{d}(X_i, X_j) = d(X_i, X_j)$. 
\end{enumerate}
\end{proposition}
\begin{proof}
    We start by proving (i). The fact that $\cS$ is a real tree can be
    checked directly from the definition. We now show that it is
    separable. Let $t \in \Q_+$, there are only countably many balls of
    $(U, d)$ of radius $t$ and positive mass, let us label them
    $(B^t_i)_{i \ge 1}$. For any $t \in \Q_+$ and $i \ge 1$, let $x^t_i
    \in B^t_i$. Let us now consider the collection $((x^t_i,t);\, t \in
    \Q_+, i \ge 1)$. First, since $\mu(B(x^t_i,t))>0$, it follows from
    the definition that $t\geq f(x^t_i)$, and thus $((x^t_i,t); t \in
    \Q_+, i \ge 1)$ is a countable collection of $\cS$ and it remains to
    show that this collection is dense in $\cS$.
    
    Let $\epsilon > 0$ and let $(x, s) \in U \times \R_+$ be in $\cS$.
    We can find $t \in \Q_+$
    such that $t > s \ge f(x)$ and $t - s < \epsilon$.
    By definition of $f$, $\mu(B(x, t)) > 0$, and we can find $i$ such that
    $B(x, t) = B^t_i$. Then $d(x, x^t_i) < t$
    and
    \[
        d(x, x^t_i) - \frac{t + s}{2} < d(x, x^t_i) - t + \frac{\epsilon}{2} < \frac{\epsilon}{2}
    \]
    and thus $d_T((x, s), (x^t_i, t)) < \epsilon$. This shows that 
    the collection is dense and that the space is separable. 

    \medskip
   
    We now turn to the proof of (ii). Let $(x,t) \in \cS$, we denote by
    \[
        C(x,t) = \Set{(y,s) \in \cS \suchthat d_T((x,t), (y,t)) = 0}
    \]
    the \emph{clade} generated by $(x,t)$. In a genealogical interpretation,
    $C(x,t)$ is the progeny of $(x,t)$ i.e.\ the subtree that has $(x,t)$ as
    its MRCA. Notice that this notion can be defined
    similarly on any rooted tree (here the root is an ``infinite point'' 
    obtained by letting $t \to \infty$). It is clear that $\psi^{-1}(C(x,t)) = B(x,t)$. 
    Our results is now immediate from the fact that the clades of a rooted separable
    tree induce the Borel $\sigma$-field of the tree. A proof of this fact is
    given in \Cref{app:tree}. 

    \medskip

    We now prove (iii). It is sufficient to prove that a.s.\ $d(X,Y) = \tilde{d}(X, Y)$
    for $X$ and $Y$ two independent variables distributed as $\mu$. Notice that for any $x, y \in U$, 
    $d(x, y) \le \tilde{d}(x, y)$. Thus the probability that 
    $d(X, Y) \ne \tilde{d}(X, Y)$ can be written
    \begin{align*} 
        \Prob{d(X,Y) \ne \tilde{d}(X, Y)} &= \iint \Indic{d(x, y) < \tilde{d}(x, y)} \mu(\mathrm{d}x)\mu(\mathrm{d}y)\\
        &= \int \mu(\mathrm{d}x) \int \mu(\mathrm{d}y) \Indic{d(x, y) < f(x)}
        = 0,
    \end{align*}
    where the last equality can be seen by writing
    \[
        \Set{x,y \in U \suchthat d(x,y) < f(x)} = 
        \bigcup_{\epsilon > 0} \Set{x,y \in U \suchthat d(x,y) < f(x) - \epsilon}
    \]
    and noticing that each event of the union in the right-hand side
    has null mass.
\end{proof}

\begin{remark}[Backbone and marked metric measure space]
    \label{rem:gufler}
    An object similar to the backbone appears in~\cite{gufler_representation_2016}
    using the framework of marked metric measure spaces introduced
    in~\cite{depperschmidt_marked_2011}. We can interpret
    the backbone as a marked metric measure space where the metric 
    space is $U$ endowed with the backbone metric
    \[
        \bar{d}(x,y) = d_S\big( (x,f(x)), (y,f(y)) \big) 
    \]
    and the mark space is $\R_+$. According to this correspondence,
    backbones are examples of elements of the set $\hat{\mathbb{U}}$
    defined in~\cite{gufler_representation_2016}.
    In~\cite{gufler_representation_2016} the marked metric measure space
    corresponding to the backbone is either considered as given, or built 
    as the completion of the ultrametric measure space on $\N$
    corresponding to the distance matrix distribution. The novelty of the
    present work is that we start from a general UMS and simplify it to
    obtain the backbone. This approach requires to identify the
    measurability assumptions to be made on UMS to avoid the problems
    that are discussed in \Cref{SS:generalUMS}.

    Moreover, the link between backbones and marked metric measure spaces
    enables us to use the work of~\cite{depperschmidt_marked_2011}.
    For instance, this provides a metric, the marked Gromov-Prohorov 
    metric, that metrizes the Gromov-weak topology on UMS and
    ensures that the topology is separable. 
\end{remark}

    \subsection{Isomorphism between backbones}
        \label{SS:samplingEquivalence}

The aim of this section is to introduce the notion of 
isomorphism between backbones and to prove our reformulation of the
Gromov reconstruction theorem.

\begin{definition}\label{def:isometry-backbone}
    Let $(U,d,\Ui,\mu)$ and $(U',d',\Ui',\mu')$ be two UMS with respective backbones
    $(\cS,\mu_{\cS})$ and $(\cS',\mu'_{\cS})$. We say that $\Phi$
    is an isomorphism from $\cS$ to $\cS'$ if:
    \begin{enumerate}
        \item[\rm (i)] The map $\Phi$ is a measure-preserving isometry from $\cS$ to $\cS'$.
        \item[\rm (ii)] For every $(x, t) \in \cS$, there exists $x' \in U'$ such that $\Phi\big((x,t)\big) = (x',t)$,
        i.e.\ $\Phi$ preserves the second coordinate. 
    \end{enumerate}
    We say that two UMS are in weak isometry when they have isomorphic backbones.
\end{definition}

Recall \Cref{prop:weakIsometry} from the introduction. We want to
show the following reformulation of \Cref{prop:weakIsometry}. In words,
this states that having the same distance matrix distribution is
equivalent to being weakly isometric.

\begin{proposition} \label{prop:samplingEquivalence}
    Let $(U,d,\Ui,\mu)$ and $(U',d',\Ui',\mu')$ be two UMS with respective backbones
    $(\cS,\mu_{\cS})$ and $(\cS',\mu'_{\cS})$. We suppose that the two backbones
    are complete metric spaces. Then the two spaces $(\cS,\mu_{\cS})$
    and $(\cS',\mu'_{\cS})$ are isomorphic \tiff the distance matrix distribution
    associated $(U, d, \Ui, \mu)$ and $(U',d', \Ui', \mu')$ are identical.
\end{proposition}

Let us compare this result to the original result
from~\cite{gromov2007metric}.  In the separable case, if two UMS share
the same coalescent then the supports of their measures are in
isometry. Thus two separable spaces that are indistinguishable in the
Gromov-weak topology share the exact same metric structure. The situation
is rather different in the general case. Even if two UMS share the same
coalescent, they can have rather different metric structures, think of
the bifurcating star-tree and the star-tree of \Cref{F:bifurcating}. What
\Cref{prop:samplingEquivalence} states is that in this case there is only
a correspondence between coarsenings of the UMS, i.e.\ the backbones on
which all the subtrees are replaced by star-trees. This result is not
surprising as the distance matrix distribution only contains the
information of a countable number of points, which is not enough to
explore the fine metric structure of the UMS.

\medskip

The ``only if'' part of \Cref{prop:samplingEquivalence} is a direct
consequence of the following lemma, which shows that the distance matrix 
distribution of a UMS can be recovered from an i.i.d.\ sequence of points of the
backbone. 
\begin{lemma} \label{lem:distanceEquivalence}
    Let $(X_i)_{i \ge 1}$ be an i.i.d.\ sequence in $U$ sampled according
    to $\mu$. Then a.s.
    \begin{align} \label{eq:backboneDistance}
        \forall i,j \ge 1,\; d(X_i, X_j) = d_{\cS}\big( (X_i, f(X_i)), (X_j, f(X_j)) \big) + \frac{f(X_i) + f(X_j)}{2}
    \end{align}
    and
    \begin{align} \label{eq:backboneDistance2}
        \forall i \ge 1,\; f(X_i) = \inf \Set{t \ge 0 \suchthat \text{$\Set{j \suchthat d(X_j, X_i) \le t}$ \mbox{\rm is infinite}}}.
    \end{align}
\end{lemma}
\begin{proof}
    We know from \Cref{prop:decomposition} that for any
    $i, j \ge 1$, $\tilde{d}(X_i, X_j) = d(X_i, X_j)$ almost surely. Suppose that
    $(X_i, f(X_i))$ and $(X_j, f(X_j))$ lie at distance $0$ in the backbone,
    then $\tilde{d}(X_i, X_j) = f(X_i) = f(X_j)$ and~\eqref{eq:backboneDistance}
    holds. Otherwise notice that $d(X_i, X_j) \ge f(X_i)$ and $d(X_i, X_j) \ge f(X_j)$.
    Thus
    \[
        d(X_i, X_j) - \frac{f(X_i) + f(X_j)}{2} \ge \frac{\Abs{f(X_i) - f(X_j)}}{2}
    \]
    and 
    \[
        d_S\big( (X_i, f(X_i)), (X_j, f(X_j)) \big) = d(X_i,X_j) - \frac{f(X_i) + f(X_j)}{2}.
    \]

    The second point of the lemma is a direct consequence of the definition
    of $f$ and of the observation that
    if $\mu(B(x, t)) > 0$, then a.s.\ there are infinitely many $(X_i)_{i \ge 1}$
    that belong to this ball.
\end{proof}

It remains to show the converse proposition, i.e.\ that if two UMS 
are sampling equivalent then they are in weak isometry. The
proof we give is an adaptation of Gromov reconstruction
theorem from Section~3.$\frac{1}{2}$.6 of~\cite{gromov2007metric}.

\begin{proof}[Proof of \Cref{prop:samplingEquivalence}]
    We say that a sequence $(x_i, t_i)_{i \ge 1}$ in $\cS$ is equidistributed
    if for any $A \in \Si$,
    \[
        \lim_{n \to \infty} \frac{1}{n} \sum_{i = 1}^n \Indic{(x_i, t_i) \in A} = \mu_{\cS}(A). 
    \]
    A well-known fact is that the empirical measure of an i.i.d.\ sample
    converges weakly to the sampling measure. Thus, a.s.\ an i.i.d.\ sequence
    is equidistributed.

    Consider the map
    \[
        D \colon
        \begin{cases}
            \cS^{\N} \to \R^{\N\times\N} \\
            (x_i, t_i)_{i \ge 1} \mapsto ( d_{\cS}\big( (x_i, t_i), (x_j, t_j) \big) + \frac{t_i + t_j}{2})_{i,j \ge 1}.
        \end{cases}
    \]
    and let $D'$ be the analogous map for $U'$. Then
    \Cref{lem:distanceEquivalence} shows that the pushforward measure
    $D \star \mu_{\cS}^{\otimes \N}$ is the distance matrix distribution associated to $U$.
    Similarly $D' \star \mu'_{\cS}{}^{\otimes \N}$ is the distance matrix distribution 
    associated to $U'$. As we have supposed that the two distance matrix distributions 
    coincide, we can find a sequence $(x_i)_{i \ge 1}$ in $U$ and a corresponding
    sequence $(x'_i)_{i \ge 1}$ in $U'$ that have the same distance matrix, i.e.\ such
    that
    \[
        D((x_i, f(x_i))_{i \ge 1}) = D'((x'_i, f(x'_i))_{i \ge 1}).
    \]
    We can suppose that these sequences are equidistributed and fulfill
    equalities~\eqref{eq:backboneDistance} and~\eqref{eq:backboneDistance2} as
    all these events have probability $1$. Using~\eqref{eq:backboneDistance2} we have 
    \[
        \forall i \ge 1,\; f(x_i) = f(x'_i)
    \]
    and then using~\eqref{eq:backboneDistance} we obtain
    \[
        \forall i,j \ge 1,\; d_{\cS}\big( (x_i, f(x_i)), (x_j, f(x_j)) \big) = 
        d'_{\cS}\big( (x'_i, f(x'_i)), (x'_j, f(x'_j)) \big).
    \]

    We now extend this correspondence to an isomorphism between the
    backbones. Let $i \ge 1$ and $t \ge f(x_i)$, we set
    \[
        \Phi( (x_i, t) ) = (x'_i, t).
    \]
    It is clear that $\Phi$ is an isomorphism from $\Set{(x_i,t)
    \in \cS \suchthat t \ge f(x_i), i \ge 1}$ to $\Set{(x'_i,t) \in \cS'
    \suchthat t \ge f(x'_i), i \ge 1}$. It is now sufficient to show that
    this set is dense to end the proof, by extending $\Phi$ to $\cS$ by
    continuity. To see that, let $(x,t) \in \cS$. As $t \ge f(x)$, we know
    that $\mu(\Set{y \in U \suchthat d(x,y) \le t+\epsilon}) > 0$ for any
    $\epsilon > 0$. Writing
        \[
        \Set{y \in U \suchthat d(x,y) \le t+\epsilon} = \Set{y \in U \suchthat 
        d_{\cS}\big((x, t+\epsilon), (y, t+\epsilon)\big) = 0},
    \]
    as $(x_i, f(x_i))_{i \ge 1}$ is equidistributed, we see that we can find $(x_i, f(x_i))$ such that 
    $(x_i, t+\epsilon) = (x, t+\epsilon)$. Moreover, it is immediate
    that $t+\epsilon \ge f(x_i)$, and we have
    \[
        d_{\cS}\big((x_i, t+\epsilon), (x, t)\big) 
        = d_{\cS}\big((x, t+\epsilon), (x, t)\big) = \epsilon.
    \]

    The fact that $\Phi$ is measure preserving holds because we have chosen equidistributed
    sequences. 
\end{proof}

\begin{remark}
    According to the correspondence between backbones and marked metric
    measure spaces outlined earlier,
    \Cref{prop:samplingEquivalence} is similar to the more general
    Theorem~1 in~\cite{depperschmidt_marked_2011}, which is itself an
    adaptation of the Gromov reconstruction theorem. However as we only
    address the case of backbones, we can be more specific. A direct
    application of Theorem~1 in~\cite{depperschmidt_marked_2011} would
    only provide an isometry between the supports of the backbones
    whereas here we obtain a global isometry.
\end{remark}

\begin{remark}
    The results of this section show that the backbone of a UMS contains
    the same information as the coalescent associated to that UMS. Thus
    properties of the coalescent can be read off from properties of the
    backbone. In particular, we can make precise an informal conjecture
    formulated in the context of exchangeable hierarchies
    in~\cite{Forman2017}, and addressed in~\cite{Forman2018}, concerning
    a nice decomposition of the sampling measure $\mu$.  Indeed, the
    sampling measure on the backbone is naturally decomposed into its
    atoms, its diffuse part on the set $\Set{(x,t) \in \cS \suchthat t =
    0}$ of leaves of $\cS$ at height $0$ and the remaining diffuse part.
    This decomposition induces three qualitatively different behaviors of
    the coalescent. In short, points sampled in the atomic part form
    singletons of the coalescent that all merge at the same time, an
    event called ``broom-like explosion'' in~\cite{Forman2017}. Second,
    points sampled in $\Set{(x,t) \in \cS \suchthat t = 0}$ always belong
    to an infinite block of the coalescent for $t > 0$, they form the
    ``iterative branching part''.  Finally points sampled in the remaining
    part of the backbone are singletons of the coalescent that
    continuously merge with existing blocks. This behavior is referred to
    as ``erosion''.
\end{remark}

    \subsection{Comb metric measure space, completion of the backbone} \label{SS:combCompletion}

An important assumption of \Cref{prop:samplingEquivalence} is that the
backbones of the UMS we consider are complete metric spaces. We will show
in this section that the UMS associated to a comb enjoys this property up
to the addition of a countable number of points. Let us start with two
examples of combs illustrating that the backbone of a comb metric measure
space is not in general complete.

First, consider the comb associated to the diadic space. Let $0 < t < 1$
and let $k$ be the only integer such that $t \in \COInterval{2^{-(k+1)}, 2^{-k}}$. 
We set
\[
    I^2_t = \bigcup_{0 \le i \le 2^{k+1}-1} \OpenInterval{i 2^{-(k+1)}, (i+1) 2^{-(k+1)}}
\]
and for $t \ge 1$ we set
\[
    I^2_t = \OpenInterval{0, 1}.
\]
The diadic comb is illustrated in \Cref{F:diadic}. Now consider the comb
metric $d^2_I$ associated to this comb, and let $x = 2^{-k}$ for some $k
\ge 1$. Consider a non-decreasing sequence $(x_n)_{n \ge 1}$ that
converges to $x$. It is not hard to see that $(x_n)_{n \ge 1}$ is Cauchy
for $d^2_I$ but does not admit a limit. 

Let us discuss a second example which is not separable. Consider the following comb 
\[
    I'_t = 
    \begin{cases}
        \emptyset &\text{ if $t < 1/2$}\\
        I^2_{t-1/2} &\text{otherwise}.
    \end{cases}
\]
This comb is illustrated in \Cref{F:diadic}. It is rather clear that the
backbone associated to $(I'_t)_{t \ge 0}$ is isometric to the backbone
obtained from $(I^2_t)_{t \ge 0}$ (notice that here the isometry is not
an isomorphism, as the backbone associated to $(I'_t)_{t \ge 0}$ is
``shifted above by $1/2$'' from that of $(I^2_t)_{t \ge 0}$). The backbone
is not complete for the same reason as above. The following proposition
shows that up to the addition of a countable number of points, we can
assume that the backbone associated to a comb metric space is complete. 

\begin{proposition} \label{prop:combCompletion}
    Consider the comb metric $d_I$ associated to a comb $(I_t)_{t \ge 0}$. 
    We can find a countable set $F$ and an extension $\bar{d}_I$ of
    $d_I$ to $[0, 1] \cup F$ such that $\bar{d}_I$ is ultrametric and the
    backbone associated to $([0, 1] \cup F, \bar{d}_I, \Is, \Leb)$ is
    complete.
\end{proposition}

\begin{remark}
    Here we have implicitly extended the Lebesgue measure to $[0, 1] \cup F$
    by giving zero mass to $F$. 
\end{remark}

A proof of this result is given in \Cref{app:combCompletion}. The 
proof of \Cref{cor:representationGeneral} now directly follows
from the various results we have shown.

\begin{proof}[Proof of \Cref{cor:representationGeneral}.]
    Let $(U, d, \Ui, \mu)$ be a UMS with complete backbone, and
    let $(\Pi_t)_{t \ge 0}$ be the associated coalescent. Using
    \Cref{TH:coalescentGeneral} we can find a nested interval-partition
    whose associated coalescent is $(\Pi_t)_{t \ge 0}$. We can now use
    \Cref{prop:combCompletion} to find a comb metric measure space whose
    backbone is complete which has the same distance matrix distribution
    as $(U, d, \Ui, \mu)$. Using \Cref{prop:samplingEquivalence} ends the
    proof. 
\end{proof}

\begin{figure}
    \includegraphics[width=\textwidth]{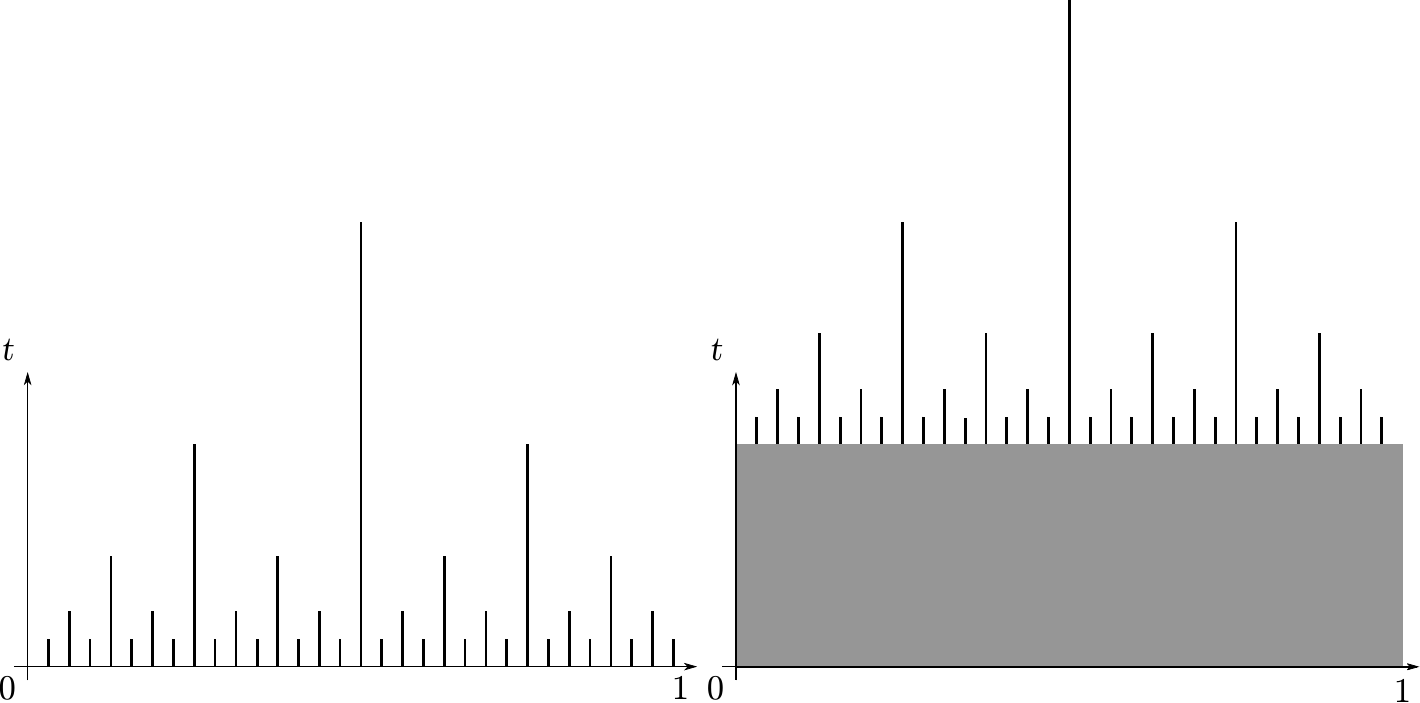}
    \caption{Left panel: The diadic comb. Right panel: The comb
    $(I'_t)_{t \ge 0}$.} \label{F:diadic}
\end{figure}

    \subsection{The separable case}
    \label{SS:separable}

In this section we consider the case of separable UMS and 
prove \Cref{cor:isometrySeparableUMS} and
\Cref{prop:representationSeparable}. The former result states that 
the weak isometry between backbones can be reinforced to an isometry
between the supports of the measures in the case of separable complete
UMS. The latter states that any complete separable ultrametric space is
isometric to a properly completed comb metric space. Let us start with
\Cref{cor:isometrySeparableUMS}.

\begin{proof}[Proof of \Cref{cor:isometrySeparableUMS}]
    Let $(I_t)_{t \ge 0}$ be a nested interval-partition without dust,
    and consider the corresponding comb metric measure space $([0, 1],
    d_I, \Is, \Leb)$. The quotient space of $\Set{f_I = 0}$ by the
    equivalence relation $x \sim y$ \tiff $d_I(x, y) = 0$ is a separable
    ultrametric space. Moreover, it is isometric to the subset $\Set{(x,
    t) \in \cS \suchthat t = 0}$ of the backbone $\cS$ of $([0, 1], d_I,
    \Is, \Leb)$. Thus the quotient space of $(\Set{f_I = 0}, d_I)$ can be
    turned into a complete ultrametric space by adding a countable number
    of points as in \Cref{prop:combCompletion}, we denote this completion
    by $(U_I, d_I)$ as in the introduction. As $(I_t)_{t \ge 0}$ has no dust, we
    have $\Leb(\Set{f_I = 0}) = 1$. Thus $U_I$ can be endowed with the
    pushforward measure of the restriction of $\Leb$ to $\Set{f_I = 0}$,
    defined on the Borel $\sigma$-field of $(U_I, d_I)$. It is a
    probability measure, let us denote it by $\Leb$. The space $(U_I,
    d_I, \Leb)$ is a separable complete Borel UMS that has the same
    distance matrix distribution as the original comb metric measure
    space $([0, 1], d_I, \Is, \Leb)$.

    Let $(U, d, \Ui, \mu)$ be a complete separable UMS. By restricting
    our attention to $\supp(\mu)$ we can assume without loss of
    generality that $\supp(\mu) = U$. According to
    \Cref{TH:combRepresentation} we can find a nested interval-partition 
    $(I_t)_{t \ge 0}$ and a corresponding comb metric measure
    space $([0, 1], d_I, \Is, \Leb)$ whose distance matrix distribution is 
    equal to that of $(U, d, \Ui, \mu)$. As $\supp(\mu) = U$, for each
    $t > 0$ we have $\mu(B(x,t)) > 0$. If $(\Pi_t)_{t \ge 0}$ denotes
    the coalescent obtained by sampling from $(U, d, \Ui, \mu)$, this
    shows that for each $t > 0$ all the blocks of $\Pi_t$ have positive
    asymptotic frequency. Thus $(I_t)_{t \ge 0}$ has no dust, and let
    $(U_I, d_I, \Leb)$ be the completion of the comb metric measure space
    as above. Then $(U, d, \mu)$ and $(U_I, d_I, \Leb)$
    are two complete separable metric measure spaces (in the usual sense)
    whose distance matrix distributions are equal. Thus, the 
    Gromov reconstruction theorem (see Section~3.$\frac{1}{2}$.6
    of~\cite{gromov2007metric}) proves that we can find a
    measure-preserving isometry between $(U_I, d_I, \Leb)$ and $(U, d,
    \mu)$, which ends the proof.
\end{proof}

We now turn to the proof of \Cref{prop:representationSeparable}. We
will need the following lemma.

\begin{lemma} \label{lem:measure}
    Any separable ultrametric space $(U, d)$ can be endowed with a measure
    $\mu$ on its Borel $\sigma$-field such that $\supp(\mu) = U$. 
\end{lemma}

\begin{proof}
    We build the measure by induction. 
    For $n = 1$, as the space is separable there are only countably many balls of radius $1$.
    If there are finitely many such balls, say $k$ balls $B_1, \dots, B_k$, 
    we define 
    \[
        \mu(B_i) = \frac{1}{k}.
    \]
    Else we can find an enumeration of the balls, $(B_i)_{i \ge 1}$, and we define 
    \[
        \mu(B_i) = \Big( \frac{1}{2} \Big)^i.
    \]
    Suppose that we have defined $\mu(B)$
    for any ball of radius $1/n$. Given a ball $B^n$ of radius
    $1/n$ there are at most countably many balls $(B^{n+1}_i)_{i \ge 1}$
    of radius $1/(n+1)$
    such that $B^{n+1}_i \subset B^n$. Similarly if there are $k$ balls
    we define
    \[
        \mu(B^{n+1}_i) = \frac{\mu(B^n)}{k}
    \]
    and if there are countably many balls we define
    \[
        \mu(B^{n+1}_i) = \mu(B^n) \Big( \frac{1}{2} \Big)^i.
    \]
    A simple application of Caratheodory's extension theorem now provides
    a probability measure $\mu$ defined on the Borel $\sigma$-field of $(U, d)$ 
    that extends this measure. It is straightforward from the construction
    that $\supp(\mu) = U$.
\end{proof}

\begin{remark} Note that a similar construction was mentioned
    in~\cite{lambert_comb_2017}, where the resulting measure was referred
    to as the ``visibility measure''.
\end{remark}

\begin{proof}[Proof of \Cref{prop:representationSeparable}]
    Let $(U, d)$ be a separable complete UMS. Using \Cref{lem:measure}
    we can find a measure $\mu$ such that $\supp(\mu) = U$. An appeal
    to \Cref{cor:isometrySeparableUMS} now proves the result.
\end{proof}

\bigskip
\noindent
{\bf Acknowledgements.} The authors thank \textit{Center for Interdisciplinary Research
in Biology} (CIRB, Coll\`ege de France) for funding. They also greatly
thank two anonymous reviewers for helpful comments.

\bibliographystyle{imsart-nameyear}
\bibliography{bibli_exchangeable_comb}

\appendix

\section{Exchangeable hierarchies} \label{app:hierarchies}

The aim of this section is to recall some results derived in~\cite{Forman2017}
and discuss the link they have with the current results.
Again, we recall that the present work should not be viewed as stemming from the work
of~\cite{Forman2017}, but should be viewed as an independent approach
bearing similarities that we now expose.

\medskip

Let $X$ be an infinite space. A hierarchy on $X$ is a collection $\Hi$ of subsets of $X$
such that
\begin{enumerate}
    \item for $x \in X$, $\Set{x} \in \Hi$, $X \in \Hi$ and $\emptyset \in \Hi$;
    \item given $A, B \in \Hi$, then $A \cap B$ is either
    $A$, $B$ or $\emptyset$.
\end{enumerate}
Any ultrametric space encodes a hierarchy that is obtained by
``forgetting the time''. More precisely, if $(U, d)$ is an ultrametric space, then 
\[
    \Hi = \Set{B(x,t),\; x \in X,\; t \ge 0} \cup \Set{\Set{x},\; x \in X} \cup \Set{X, \emptyset}
\]
is a hierarchy. The hierarchy $\Hi$ encodes the genealogical structure of $(U, d)$,
i.e.\ the order of coalescence of the families, but not the coalescence times.
\begin{remark}
    The converse does not hold, there exist hierarchies that cannot be obtained as the
    collection of balls of an ultrametric space. 
    For example, consider a space $X$ with cardinality greater than
    the continuum, endowed with a total order $\le$, and define 
    \[
        \Hi = \Set{ \Set{y \suchthat y \le x} \suchthat x \in X}     
        \cup \Set{\Set{x},\; x \in X} \cup \Set{X, \emptyset}.
    \]
\end{remark}

The main object studied in~\cite{Forman2017} are exchangeable hierarchies on 
$\N$. Let $\sigma$ be a permutation of $\N$, and $\Hi$ be a hierarchy on
$\N$. Then $\sigma$ naturally acts on $\Hi$ as
\[
    \sigma(\Hi) = \Set{\sigma(A),\; A \in \Hi}.
\]
A random hierarchy on $\N$ (see~\cite{Forman2017} for a definition of 
the $\sigma$-field associated to hierarchies) is called exchangeable
if for any permutation $\sigma$,
\[
    \sigma(\Hi) \overset{(\mathrm{d})}{=} \Hi.
\]
In a similar way that exchangeable coalescents are
obtained by sampling in UMS, exchangeable hierarchies are
obtained by sampling in hierarchies on measure spaces. 
Let $(X, \mu)$ be a probability space, and consider a hierarchy
$\Hi$ on $X$. An exchangeable hierarchy $\Hi'$ can be generated out of an 
i.i.d.\ sequence $(X_i)_{i \ge 1}$ by defining
\[
    \Hi' = \Set{ \Set{i \ge 1 \suchthat X_i \in A},\; A \in \Hi}.
\]
Again, an exchangeable hierarchy can be obtained from an exchangeable
coalescent by forgetting the time. Let $(\Pi_t)_{t \ge 0}$ be an
exchangeable coalescent. Then
\[
    \Hi = \Set{ B,\; \text{$B$ is a block of $\Pi_t$, $t \ge 0$} }
\]
is an exchangeable hierarchy. 

The main results in~\cite{Forman2017} show that any exchangeable 
hierarchy can be obtained by sampling from 1) a random ``interval hierarchy''
on $\COInterval{0, 1}$ and 2) a random real-tree. The link with our results
now seems straightforward. 

\medskip

An interval hierarchy on $\COInterval{0, 1}$ is a hierarchy $\Hi$
on $\COInterval{0, 1}$ such that all non-singleton elements
of $\Hi$ are intervals. Again, an interval hierarchy can be
obtained from a nested interval-partition $(I_t)_{t \ge 0}$
by forgetting the time. The family of sets
\begin{align*}
    \Hi = &\Set{I \suchthat \text{$I$ is an interval component of $I_t$}, t \ge 0}\\ 
          &\cup \Set{ \Set{x}, x \in \COInterval{0, 1} } \\
          &\cup \Set{\COInterval{0, 1}, \emptyset}
\end{align*}
is an interval hierarchy. Theorem~4 in~\cite{Forman2017} states that any exchangeable
hierarchy on $\N$ can be obtained by sampling in a random
interval hierarchy. This is the direct equivalent of our
\Cref{TH:coalescentGeneral} that states that any exchangeable coalescent
can be obtained by sampling in a random nested interval-partition. 

\medskip

Consider a measure rooted real-tree $(T, d, \rho, \mu)$,
it can be endowed with a partial order $\preceq$ such
that $y \preceq x$ if $x$ is an ancestor of $y$
(see~\cite{evans_probability_2007}). Then, the fringe subtree
of $T$ rooted at $x \in T$ is defined as the set
\[
    F_T(x) = \Set{y \in T \suchthat y \preceq x},
\]  
it is the set of the offspring of $x$. 
The natural hierarchy associated to $(T, d, \rho)$ is
\[
    \Hi = \Set{F_T(x),\; x \in T}. 
\]
Theorem~5 in~\cite{Forman2017} states that any exchangeable
hierarchy can be obtained by sampling in the hierarchy
associated to a random measure rooted real-tree. 
In our framework, we have seen that a nested interval-partition 
can be seen as an ultrametric space, and in \Cref{SS:combCompletion}
we have seen how this ultrametric space is embedded in a real-tree. 
Again we have proved here the reformulation of Theorem~5 from~\cite{Forman2017}.

\medskip

In a subsequent work, one of the authors has introduced the notion
of mass-structural isomorphism \citep{Forman2018}. In a nutshell,
two trees that are mass-structural isomorphic induce the same
exchangeable hierarchy. In our framework, two spaces have 
the same coalescent \tiff their backbones are isomorphic.
Thus, the mass-structural isomorphism is replaced here by the simpler
notion of isomorphism.

\medskip

Overall, the two works are very similar in the sense that they obtain 
the same kind of representation results for exchangeable hierarchies
and exchangeable coalescents. However the techniques used in the proofs
are different, e.g.\ the work of~\cite{Forman2017} relies on spinal 
decomposition whereas the present work relies on nested compositions. 
Moreover, as an ultrametric space contains ``more information'' than
a hierarchy, our results are not trivially implied by the results 
in~\cite{Forman2017}, but constitute an extension of their work.

Finally, we wish to stress two things. First, most of the difficulties
that \Cref{S:ultrametric} deal with stem from the fact that we
consider non-separable metric spaces. These issues and the work that is
done here heavily relies on the theory of metric spaces. Seeing genealogies
as metric spaces is only possible if we keep the information on the times of coalescence,
which is not the case when considering hierarchies.

Second, keeping this information allows us to study genealogies
as time-indexed stochastic processes. It is a necessary step to study
the Markov property of the combs associated to $\Lambda$-coalescents
as in \Cref{S:lambda}. This creates a direct link between the 
present work and the very rich litterature on $\Lambda$-coalescents
and coalescence theory that is not present in~\cite{Forman2017}. Moreover,
this provides a new approach to the question of dynamical genealogies,
with the introduction of the dynamical comb.

\section{Independence of the nested interval-partitions and the sampling variables}
\label{app:detailsFinetti}

Consider an exchangeable nested composition $(\Ci_t)_{t \ge 0}$,
and let $(I_t)_{t \in \Q_+}$ be the nested interval-partition
obtained by applying \Cref{TH:gnedin} distinctly for any
$t \in \Q_+$, and $(V_i)_{i \ge 1}$ be the sequence of i.i.d.\ uniform
variables obtained from \Cref{TH:gnedin} applied at time $0$. 
The aim of this section is to show that $(V_i)_{i \ge 1}$ is independent
from $(I_t)_{t \in \Q_+}$. 

Let $0 = t_0 < t_1 < \dots < t_p$. We can build a collection of sequences 
$(\xi^{(k)}_i)_{i \ge 1, k = 0, \dots, p}$ where for $k = 0, \dots, p$
and $i \ge 1$,
\[
    \xi^{(k)}_i = \lim_{n \to \infty} \frac{1}{n} \sum_{j = 1}^n \Indic{j \preceq_k i},
\]
and $\preceq_k$ is the partial order on $\N$ representing $\Ci_{t_k}$ as in 
\Cref{SS:composition}. The sequence of vectors $(\xi^{(0)}_i, \dots, \xi^{(p)}_i)_{i \ge 1}$
is exchangeable. Thus by applying a vectorial version of de Finetti's theorem
we know that there exists a measure $\mu$ on $[0, 1]^{p+1}$ such that conditionally
on $\mu$ the sequence of vectors is i.i.d.\ distributed as $\mu$. We can now
``spread'' the variables $(\xi^{(0)}_i)_{i \ge 1}$ using an independent i.i.d.\ uniform
sequence as in the proof of \Cref{TH:gnedin} to obtain a sequence
$(V_i)_{i \ge 1}$ that is i.i.d.\ uniform conditionally on $\mu$. Thus the 
sequence $(V_i)_{i \ge 1}$ is independent of $\mu$. The interval-partitions
$(I_{t_0}, \dots, I_{t_p})$ can be recovered from the push-forward measures of
$\mu$ by the coordinate maps on $\R^{p+1}$. Thus $(V_i)_{i \ge 1}$ is independent from $(I_t)_{t \in \Q_+}$.

\section{Generator calculation}
\label{app:generator}

Let $n \ge 1$ and let $\hat{Q}_n$ denote the generator of the nested
composition $(\Ci^n_t)_{t \ge 0}$ defined from the transition rates
$(\tl_{b,k};\, 2 \le k \le b < \infty)$.  Let $Q_n$ be the generator of
the restriction to $[n]$ of a $\Lambda$-coalescent.  Here we show that
for any function $f$, from the space of compositions of $[n]$ to $\R$,
\[
    \forall \pi,\; \hat{Q}_n Lf(\pi) = LQ_n f(\pi),
\]
where $L$ is the operator defined in \Cref{S:lambda}. 

We will need additional notations. The space of partitions and
compositions of $[n]$ will be denoted by $P_n$ and $S_n$ respectively. For
$\pi, \pi' \in P_n$, we denote by $q_{\pi, \pi'}$ the transition rate from
$\pi$ to $\pi'$, i.e.\ $q_{\pi,\pi'} = \lambda_{b,k}$ if $\pi$ has $b$
blocks and $\pi'$ is obtained by merging $k$ blocks of $\pi$, and
$q_{\pi,\pi'} = 0$ otherwise.  Similarly for $c, c' \in S_n$ we define
$q_{c, c'}$ to be the transition rate from $c$ to $c'$. Finally, we
denote by $O(\pi)$ the set of compositions of $[n]$ whose blocks are given
by the partition $\pi$, and $\Card(\pi)$ the number of blocks of $\pi$.
Let $\pi \in P_n$ and denote by $b$ the number of blocks of $\pi$, we have
\begin{align*}
    \hat{Q}_n Lf(\pi) &= \sum_{\pi' \in P_n} q_{\pi,\pi'} (Lf(\pi') - Lf(\pi)) \\
                    &= \sum_{\pi' \in P_n} q_{\pi,\pi'} \Big(\sum_{c' \in O(\pi')} \frac{1}{\Card(\pi')!} f(c') - 
                    \sum_{c \in O(\pi)} \frac{1}{\Card(\pi)!} f(c)\Big) \\
                    &= \sum_{\pi' \in P_n} \sum_{c' \in O(\pi')} q_{\pi,\pi'} \frac{1}{\Card(\pi')!} f(c') - 
                    \sum_{c \in O(\pi)} \sum_{k = 2}^b \frac{1}{\Card(\pi)!} \binom{b}{k} \lambda_{b,k} f(c).
\end{align*}
Similarly, we have
\begin{align*}
    LQ_n f(\pi) &= \sum_{c \in O(\pi)} \frac{1}{\Card(\pi)!} Q_n f(c) \\
             &= \sum_{c \in O(\pi)} \frac{1}{\Card(\pi)!} \sum_{c' \in S_n} q_{c,c'} (f(c') - f(c)) \\
             &= \sum_{c \in O(\pi)} \frac{1}{\Card(\pi)!} \sum_{c' \in S_n} q_{c,c'} f(c') 
              - \sum_{c \in O(\pi)} \sum_{k = 2}^b \frac{1}{\Card(\pi)!} \binom{b}{k} \lambda_{b,k} f(c).
\end{align*}
We will end the calculation by showing that for any $c' \in S_n$, the coefficient
in front of the term $f(c')$ in the left sum is the same for both expression.
Let $\pi'$ be the partition associated to $c'$. If $\pi'$ is not obtained
by merging $k$ blocks of $\pi$ for some $k$, then the coefficient of the term $f(c')$
in the sum is $0$ in both expressions. Now suppose that $\pi'$ is obtained by merging
$k$ blocks of $\pi$. In the first expression, we first choose the blocks of $\pi$ that merge to get $\pi'$
and then order the resulting partition to get the composition $c'$. There is only one possible
way to do that and obtain a given $c'$. Thus the coefficient in front of
$f(c')$ is $\lambda_{b,k} / (b-k+1)!$. In the second expression, we first choose an
order to obtain a composition $c$, and then merge its blocks to get the composition $c'$.
There are $k!$ possible orderings of $\pi$, and then exactly one merger of $c$ that
lead to $c'$ (we can take any permutation of the $k$ blocks that merge). 
Thus the coefficient in front of term $f(c')$ is 
\[
    \frac{k!}{b!} \tl_{b,k} = \frac{k!}{b!} \frac{1}{b-k+1} \frac{b!}{k!\,(b-k)!} \lambda_{b,k} = \frac{1}{(b-k+1)!} \lambda_{b,k}.
\]

\section{Measurability of separable rooted trees}
\label{app:tree}

In this section we prove the claim made in the proof of
\Cref{prop:decomposition} that the Borel $\sigma$-field of a
separable rooted tree is induced by the clades of the tree. Let us be
more specific.

We consider a separable real-tree $(T, d)$ with a particular point $\rho
\in T$ that we call the root. For $x,y \in T$, we denote by $[x, y]$ the
unique geodesic with endpoints $x$ and $y$
(see~\cite{evans_probability_2007}). Recall from 
\Cref{app:hierarchies} the fringe subtree of $T$ rooted at $x$
equivalently defined as the \emph{clade}
\[
    C(x) = \Set{y \in T \suchthat x \in [\rho, y]},
\]
see \Cref{F:clade} for an illustration. The claim is that
\[
    \sigma(\Set{C(x),\; x \in T}) = \Bi(T).
\]
\begin{remark}
Our goal in the proof of \Cref{prop:decomposition} is to apply the result to the backbone whose root should be such that clades are the balls of $U$.  This can be done by seeing the backbone
    as having a root ``at infinity''.
\end{remark}

Let $x \in T$ and $\epsilon > 0$, we assume that 
$\epsilon < d(x, \rho)$. We denote by $B(x, \epsilon)$ the
open ball centered in $x$ with radius $\epsilon$, and 
$S(x, \epsilon)$ the sphere of center $x$ and radius
$\epsilon$, i.e.\
\[
    S(x, \epsilon) = \Set{y \in T \suchthat d(x,y) = \epsilon}.
\]
There is a unique point in $a \in [\rho, x] \cap S(x, \epsilon)$. It
is clear that
\[
    B(x, \epsilon) = C(a) \setminus \bigcup_{y \in S(x, \epsilon) \setminus \Set{a}} C(y).
\]
Let $y \in S(x,\epsilon)$, and $0 < \eta < \epsilon$, 
we denote by $y_\eta$ the only point in $[y, x]$ such that $d(y_\eta, y) = \eta$. 
We can write
\[
    \bigcup_{y \in S(x, \epsilon) \setminus \Set{a}} C(y) = \bigcap_{\eta > 0} \bigcup_{y \in S(x, \epsilon) \setminus \Set{a}} C(y_\eta).
\]
The claim is proved if we can show that the union on the right-hand side is
countable. This holds due to the separability of $(T, d)$. To see that
notice that by uniqueness of the geodesic, if $y$ and $y'$ are
such that $y_\eta \ne y'_\eta$, then $d(y, y') > \eta$. Thus
if the set $\Set{y_\eta \suchthat y \in S(x, \epsilon) \setminus \Set{a}}$ is not countable,
we can find an uncountable subset of $S(x,\epsilon)$ such that any two points
lie at distance at least $\eta$. This is not possible due to separability. 

\begin{figure}
    \center
    \includegraphics{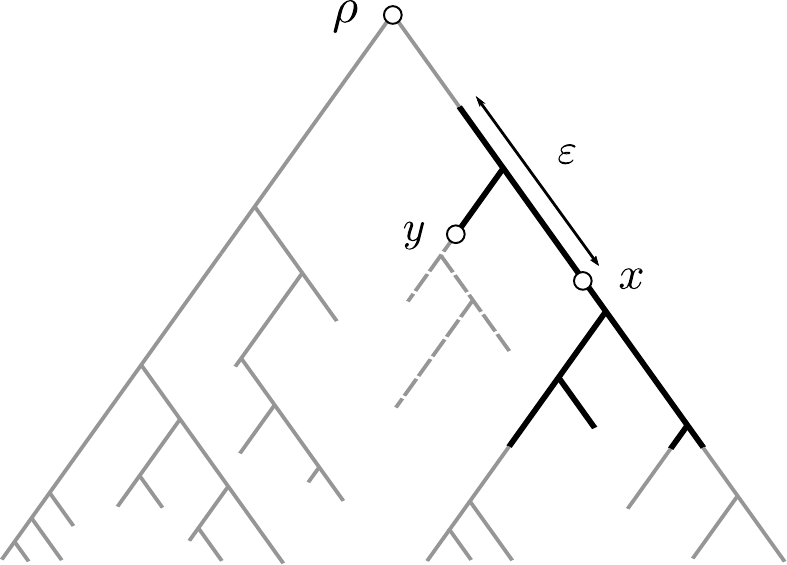}
    \caption{A tree rooted at $\rho$. The ball of radius $\epsilon$ and center
    $x$ is represented by the black bold lines. An example of $y \in S(x, \epsilon)$ is
    given, and its corresponding clade $C(y)$ is represented by grey dashed lines.} \label{F:clade}
\end{figure}

\section{Comb completion}
\label{app:combCompletion}

In this section we prove \Cref{prop:combCompletion}, i.e.\ that
the backbone of a comb is complete up to the addition of a countable number
of points. We start
from a nested interval-partition $(I_t)_{t \ge 0}$. 
We define
\begin{multline*}
    \Ri = \{x \in [0, 1] : \exists s_x, t_x
    \text{ s.t.\ $x$ is the right endpoint}\\
    \text{of an interval component of $I_u$
    for $u \in [s_x,t_x]$}\}
\end{multline*}
and 
\begin{multline*}
    \Li = \{x \in [0, 1] : \exists s_x, t_x
    \text{ s.t.\ $x$ is the left endpoint}\\
    \text{of an interval component of $I_u$ %
    for $u \in [s_x,t_x]$}\}.
\end{multline*}

We now work with a subset of $[0, 1] \times \Set{0, r, \ell}$. Let 
\[
    \bar{I} = ([0, 1] \times \Set{0}) \cup (\Ri \times \Set{r}) \cup (\Li \times \Set{\ell}).
\]
We will simply write $x$ for $(x, 0)$,
$x_r$ for $(x, r)$ if $x \in \Ri$ and $x_\ell$ for $(x,\ell)$ if $x \in \Li$. 
We extend $d_I$ to $\bar{I}$ in the following way. Let $x < y$, we
define
\begin{gather*}
    \bar{d}_I(x, y) = \bar{d}_I(x, y_\ell) = \bar{d}_I(x_r, y_\ell) = \bar{d}_I(x_r, y) = \sup_{\ClosedInterval{x,y}} f_I \\
    \bar{d}_I(x, y_r) = \bar{d}_I(x_r, y_r) = \sup_{\COInterval{x,y}} f_I \\
    \bar{d}_I(x_\ell, y) = \bar{d}_I(x_\ell, y_\ell) = \sup_{\OCInterval{x,y}} f_I \\ 
    \bar{d}_I(x_\ell, y_r) = \sup_{\OpenInterval{x,y}} f_I
\end{gather*}
and $\bar{d}_I(x_r, x_\ell) = f(x)$. We use symmetrized definitions if $x > y$.
It is straightforward to check that $\bar{d}_I$ is a pseudo-ultrametric.
We will denote by $\cS_I$ the backbone associated to this UMS, and $d_{\cS_I}$ the restriction
of the tree metric to $\cS_I$, i.e.
\[
    \forall (x',t),\; (y',s) \in \cS_I,\; d_{\cS_I}\big( (x', t), (y', s) \big) = 
    \max\Set{\bar{d}_I(x', y') - \frac{t+s}{2}, \frac{\Abs{t - s}}{2}}.
\]

\begin{lemma}
    The backbone $(\cS_I, d_{\cS_I}, \Leb)$ associated to $(\bar{I}, \bar{d}_I, \Leb)$
    is a complete metric space.
\end{lemma}

\begin{proof}
    Consider $(x'_n, t_n)_{n \ge 1}$ a Cauchy sequence in $\bar{I}$ for
    the metric $d_{\cS_I}$. As
    \[
        \frac{\Abs{t_n - t_m}}{2} \le d_{\cS_I}\big( (x'_n, t_n), (x'_m, t_m) \big),
    \]
    the sequence $(t_n)_{n \ge 1}$ is Cauchy and converges to a limit
    that we denote by $t$. Each point $x'_n$ can be written as $x'_n =
    (x_n, a_n)$ with $x_n \in [0, 1]$ and $a_n \in \Set{0, r, \ell}$. The
    sequence $(x_n)_{n \ge 1}$ admits a subsequence that converges to a
    limit $x$ for the usual topology in $[0, 1]$. Without loss of
    generality we can assume that $(x_n)_{n \ge 1}$ is non-decreasing and
    converges to $x$. 

    Using the fact that the sequence is Cauchy, we know that
    \[
        \lim_{n \to \infty} \sup_{m \ge n} \bar{d}_I(x'_n, x'_m) - \frac{t_n + t_m}{2} \le 0,
    \]
    which directly implies that
    \[
        \lim_{\epsilon \to 0} \sup_{\COInterval{x - \epsilon, x}} f_I \le t.
    \]

    Suppose that $x \in \Ri$. By definition of $\bar{d}_I$ and the above
    remark,
    \[
        \lim_{n \to \infty} \bar{d}_I(x'_n, x_r) - \frac{t_n + t}{2} \le 0.
    \]
    Thus the sequence $(x'_n, t_n)_{n \ge 1}$ converges to $(x_r, t)$.

    Now suppose that $x \not\in \Ri$. We claim that
    \[
        \lim_{\epsilon \to 0} \sup_{\COInterval{x-\epsilon, x}} f_I = f(x).
    \]
    As $x \not\in \Ri$ we directly know that
    \[
        \lim_{\epsilon \to 0} \sup_{\COInterval{x-\epsilon, x}} f_I \ge f(x).
    \]
    Suppose that the above limit is strictly greater than $f(x)$. Then we
    can find a non-decreasing sequence $(y_n)_{n \ge 1}$ converging to
    $x$ in the usual topology such that $f(y_n) \downarrow \lambda > f(x)$ as
    $n$ goes to infinity.  Let $\eta < \lambda - f(x)$.  Notice that the set
    $\Set{y \in [0, 1] \suchthat f(y) > \lambda - \eta}$ is closed in the usual
    topology, as it is the complement of $I_{\lambda-\eta}$. This shows that
    $x$ belongs to this set, which is a contradiction. Our claim is
    proved. Similarly to above, it is now immediate that 
    \[
        \lim_{n \to \infty} \bar{d}_I(x'_n, x) - \frac{t_n + t}{2} \le 0.
    \]
    and that $(x'_n, t_n)_{n \ge 1}$ converges to $(x, t)$. 
\end{proof}

\begin{remark}
    This completion is already present in the compact case 
    in~\cite{lambert_comb_2017}. In this case, we have 
    $\Ri = \Li = \Set{f_I > 0}$.
\end{remark}

    \section{The link between dust and the Banach-Ulam problem} 
    \label{app:banachUlam}

In this section we prove \Cref{prop:banachUlam}. We prove this
result by constructing a solution to the so-called Banach-Ulam problem.
This problem can be formulated as follows: is it possible to find a 
space $X$ with a probability measure $\mu$ on the power-set
$\mathcal{P}(X)$ of $X$ such that $\mu(\Set{x}) = 0$ for all $x \in X$?

Recall that a UMS $(U, d, \Ui, \mu)$ is called a Borel UMS if 
$\Ui$ is the Borel $\sigma$-field of $(U, d)$. The support of the measure
$\mu$, $\supp(\mu)$, is defined as the intersection of all balls with positive mass.
Equivalently, it can be defined as 
\[
    \supp(U) = \Set{x \in U \suchthat \forall t > 0,\; \mu(B(x, t)) > 0}.
\]
We start with the following lemma, which gives a necessary and sufficient
condition for the coalescent sampled from $U$ to have dust in terms of
the support of $\mu$.

\begin{lemma} \label{lem:dustSupport}
    Let $(U, d, \Ui, \mu)$ be a UMS, and let $(\Pi_t)_{t \ge 0}$ be the
    associated coalescent. Then $(\Pi_t)_{t \ge 0}$ has dust \tiff
    $\mu(\supp(\mu)) < 1$. 
\end{lemma}

\begin{proof}
    Let $(X_i)_{i \ge 1}$ be an i.i.d.\ sequence in $U$ distributed as $\mu$
    and let $(\Pi_t)_{t \ge 0}$ be the coalescent obtained as above.
    We say that $i$ is in the dust of the coalescent if there exists
    $t > 0$ such that $\Set{i}$ is a singleton block of $\Pi_t$. 
    We show that a.s.
    \[
        \text{$i$ is in the dust} \iff X_i \not\in \supp(\mu).
    \]
    Suppose that $X_i \in \supp(\mu)$. Then for any $t > 0$, $\mu(B(X_i,
    t)) > 0$, thus a.s.\ there are infinitely many other variables
    $(X_j)_{j \ge 1}$ in $B(X_i, t)$. Thus $X_i$ is in an infinite block
    of $\Pi_t$. Conversely suppose that $i$ is not in the dust, i.e.\
    that for any $t > 0$, $\Set{i}$ is not a singleton block.  Using
    Kingman's representation theorem for exchangeable partitions, we know
    that the block of $i$ is a.s.\ infinite and has a positive asymptotic
    frequency $f_i$. The law of large numbers shows that $f_i =
    \mu(B(X_i, t)) > 0$.
\end{proof}

\begin{proof}[Proof of \Cref{prop:banachUlam}]
    Let us start by showing that (i) implies (iii). Let $(U, d, \Ui,
    \mu)$ be a Borel UMS with associated coalescent $(\Pi_t)_{t \ge 0}$.
    Suppose that $(\Pi_t)_{t \ge 0}$ has dust. According to
    \Cref{lem:dustSupport}, we have $\mu(\supp(\mu)) < 1$.
    Consider $t > 0$ and let
    $(B^{t}_{\alpha})_{\alpha \in A_t}$ be the collection of open balls
    of radius $t$ with zero mass, where $A_t$ is just an index set. We
    know that
    \[
        \bigcup_{t > 0} \bigcup_{\alpha \in A_t} B^t_\alpha = U \setminus \supp(\mu). 
    \]
    Using the continuity from below of the measure $\mu$, we can find an
    $\epsilon > 0$ such that $\mu(\bigcup_{\alpha \in A_\epsilon} B^\epsilon_\alpha) > 0$. 
    We now consider the equivalence relation
    \[
        x \sim y \iff d(x,y) < \epsilon
    \]
    and denote by $X$ the quotient space of 
    $\bigcup_{\alpha \in A_\epsilon} B^\epsilon_\alpha$ for the relation
    $\sim$. We define the quotient map as
    \[
        \phi \colon \begin{cases}
            U \to X \\
            x \mapsto \Set{y \in U \suchthat d(x,y) < \epsilon}.
        \end{cases}
    \]
    We claim that $\phi$ is continuous when
    $U$ is equipped with the metric topology induced by $d$, and 
    $X$ is equipped with the discrete topology $\mathcal{P}(X)$. 
    Let $C \subset X$, then
    \[
        \phi^{-1}(C) = \bigcup_{x \in \phi^{-1}(C)} B(x, \epsilon) 
    \]
    which is an open subset of $U$. We call $\mu_X$ the push-forward 
    measure of $\mu$ by the map $\phi$. The measure
    $\mu_X / \mu_X(X)$ is a diffuse probability measure defined 
    on $\mathcal{P}(X)$ as required. Thus, $(X, \mathcal{P}(X), \mu_X)$
    is a solution to the Banach-Ulam problem.

    Using the terminology from~\cite{fremlin_real_1993}, this
    proves that the cardinality of $X$ is a real-valued cardinal
    (see Notation~1C in~\cite{fremlin_real_1993}). According to Ulam's
    theorem (see Theorem~1D in~\cite{fremlin_real_1993}), real-valued
    cardinals fall into two classes: atomlessly-measurable cardinals and 
    two-valued-measurable cardinals. The cardinal of $X$ is
    atomlessly-measurable. To see this, one can for example notice that
    our measurability assumption on $d$ implies that the cardinality of $U$
    (and thus that of $X$) is not larger than the continuum. (If this does
    not hold, then the diagonal does not belong to the product $\sigma$-field
    $\mathcal{P}(U) \otimes \mathcal{P}(U)$ and the metric $d$ is not
    measurable.) Finally, using Theorem~1D of~\cite{fremlin_real_1993} proves
    (iii).

    \medskip

    The fact that (ii) implies (i) is obvious, it remains to show that 
    (iii) implies (ii). Suppose that there exists an extension
    of the Lebesgue measure to all subsets of $\R$, let us denote by
    $\overline{\Leb}$ its restriction to $[0, 1]$. Let $(\Pi_t)_{t \ge
    0}$ be any coalescent with dust. By \Cref{TH:coalescentGeneral} we
    can find a nested interval-partition $(I_t)_{t \ge 0}$ such that 
    the paintbox based on $(I_t)_{t \ge 0}$ is distributed as 
    $(\Pi_t)_{t \ge 0}$. Let $d_I$ be the corresponding comb metric on
    $[0, 1]$. Then $([0, 1], d_I, \Bi_I([0, 1]), \overline{\Leb})$
    is a UMS, where $\Bi_I([0, 1])$ refers to the Borel $\sigma$-field
    induced by $d_I$ and $\overline{\Leb}$ is restricted to that
    $\sigma$-field. The coalescent obtained by sampling from this UMS is
    distributed as $(\Pi_t)_{t \ge 0}$.
\end{proof}

\end{document}